\documentclass[12pt,reqno]{amsart}
\usepackage{amsmath}
\usepackage{amsfonts}
\usepackage{amssymb}
\usepackage{mathtools}
\usepackage{graphicx}
\usepackage{color}
\usepackage{hyperref}
\usepackage{geometry}
\usepackage{tikz-cd}
\usepackage{enumerate}

\usepackage{xcolor}
\newtheorem{theorem}{Theorem}[section]
\newtheorem*{theorem*}{Theorem}
\newtheorem{lemma}[theorem]{Lemma}
\newtheorem{corollary}[theorem]{Corollary}

\theoremstyle{definition}
\newtheorem{definition}[theorem]{Definition}

\newtheorem{remark}{Remark}[section]

\DeclareMathOperator{\trace}{tr} 
\DeclareMathOperator{\di}{div}

\DeclareMathOperator{\rank}{rank}

\def\rr{{\mathbb R}}
\def\cc{{\mathbb C}}

\def\db{{\partial_b }}
\def\dbb{{\bar{\partial}_b }}
\def\lap{{\Delta}}

\newcommand{\la}{\langle}
\newcommand{\ra}{\rangle}

\allowdisplaybreaks

\begin{document}

\title{On $\overline{\partial }_{b}$-harmonic maps from pseudo-Hermitian
manifolds to K\"{a}hler manifolds}

\author{Yuxin Dong}
\address{School of Mathematical Science  and  Laboratory of Mathematics for Nonlinear Science, Fudan University, Shanghai 200433, PR China}
\email{yxdong@fudan.edu.cn}

\author{Hui Liu}
\address{School of Mathematical Sciences, Fudan University, Shanghai 200433, PR China}
\email{21110180015@m.fudan.edu.cn and liuh45@univie.ac.at}

\author{Biqiang Zhao}
\address{Beijing International Center for Mathematical Research, Peking University, Beijing, 100871, PR China}
\email{2306394354@pku.edu.cn}

\begin{abstract}
    In this paper, we consider maps from pseudo-Hermitian manifolds to K\"{a}hler manifolds and introduce partial energy functionals for these maps.
    First, we obtain a foliated Lichnerowicz type result on general
    pseudo-Hermitian manifolds, which generalizes a related result on Sasakian
    manifolds in \cite{SSZ2013holomorphic}. Next, we investigate critical maps of the partial
    energy functionals, which are referred to as $\overline{\partial }_{b}$-harmonic maps and $\partial _{b}$-harmonic maps. We give a foliated result
    for both $\overline{\partial }_{b}$- and $\partial _{b}$-harmonic maps,
    generalizing a foliated result of Petit \cite{Pet2002harmonic} for harmonic maps. Then we
    are able to generalize Siu's holomorphicity result for harmonic maps \cite{Siu1980rigid}
    to the case for $\overline{\partial }_{b}$- and $\partial _{b}$-harmonic
    maps.
    
    ~

    \paragraph{\textbf{Mathematics Subject Classification (2020)}:} 53C25, 58E20.
\end{abstract}
\maketitle

\renewcommand{\thefootnote}{\fnsymbol{footnote}}

\footnotetext{\hspace*{-5mm} \begin{tabular}{@{}r@{}p{13.4cm}@{}}
& This work was supported by  NSFC Grant (No. 12171091) and  the China Scholarship Council (No. 202306100156).\\
\end{tabular}}

\section{Introduction}
In \cite{Siu1980rigid}, Y.T. Siu proved the following theorem : \emph{Let }$%
f:M\rightarrow N$\emph{\ be a harmonic map between compact K\"{a}hler
manifolds. If }$(N,g)$\emph{\ has strongly negative curvature and }$%
\rank_{\rr}(df_{x})\geq 4$\emph{\ at some point }$x\in M$\emph{, then }$f$\emph{\
is holomorphic or anti-holomorphic}.\ 

The above theorem, combined with Eells-Sampson's existence theorem \cite{ES1964harmonic},
implies Siu's celebrated strong rigidity for compact K\"{a}hler manifolds
with strongly negative curvature.  Subsequently, there have been some
research efforts to generalize Siu's theorem to the case of non-K\"{a}hler
Hermitian manifolds. In \cite{JY1993nonlinear}, Jost and Yau used Hermitian harmonic maps to
generalize Siu's rigidity theorem to the case where the domain manifold is
astheno-K\"{a}hler. In \cite{LY2014hermitian}, Liu and Yang considered the critical points
of partial energies for maps from Hermitian manifolds, and discussed related
holomorphicity results for these critical maps.

A pseudo-Hermitian manifold $(M^{2m+1},H,J,\theta )$ is a strictly
pseudoconvex CR manifold $(M,H,J)$ endowed with a pseudo-Hermitian structure 
$\theta $. It can be regarded as an odd dimensional analogue of a Hermitian
manifold. Harmonic maps and their generalizations have also been used to
study pseudo-Hermitian manifolds. In \cite{Pet2002harmonic}, Petit established some rigidity
results for harmonic maps from pseudo-Hermitian manifolds. First, he proved
that any harmonic map from a compact Sasakian manifold to a Riemannian
manifold with nonpositive sectional curvature is trivial on the Reeb field
of the pseudo-Hermitian structure. A map with this property is said to be
foliated. Next he proved that under a similar rank condition as above, the
harmonic map from a compact Sasakian manifold to a K\"{a}hler manifold with
strongly negative curvature is CR-holomorphic or CR-antiholomorphic. In
\cite{CDRY2019onharm}, among other results, the authors generalized Petit's results to
the case of pseudoharmonic maps. Besides, Li and Son \cite{LS2019CRanalogue} defined the
following $\overline{\partial }_{b}$-energy functional for maps from a
pseudo-Hermitian manifold to a K\"{a}hler manifold:
\[
E_{\overline{\partial }_{b}}(f)=\frac{1}{2}\int_{M}|\overline{\partial }%
_{b}f|^{2}dv_{\theta }.
\]%
The $\partial _{b}$-energy functional $E_{\partial _{b}}(f)$ can be defined
similarly. A critical point of $E_{\overline{\partial }_{b}}(\cdot )$ was
called pseudo-Hermitian harmonic.\ Then they proved a "Siu-type
holomorphicity" result for a pseudo-Hermitian harmonic map under a rank
condition on a dense subset of $M$. 

In this paper, we consider maps from a pseudo-Hermitian manifold $M$ to a K\"{%
a}hler manifold $(N,\widetilde{J},\widetilde{g})$, and introduce the
following partial energy functionals  
\begin{equation}
E_{\overline{\partial }_{b},\xi }(f)=\frac{1}{2}\int_{M}\left\{ |\overline{%
\partial }_{b}f|^{2}+\frac{1}{4}|df(\xi )|^{2}\right\} dv_{\theta } 
\end{equation}%
and 
\begin{equation}
E_{\partial _{b},\xi }(f)=\frac{1}{2}\int_{M}\left\{ |\partial _{b}f|^{2}+%
\frac{1}{4}|df(\xi )|^{2}\right\} dv_{\theta }.  
\end{equation}%
where $\xi $ denotes the Reeb vector field of $(M,\theta )$. Note that the
usual energy $E(f)=E_{\overline{\partial }_{b},\xi }(f)+E_{\partial _{b},\xi
}(f)$. A critical point of $E_{\overline{\partial }_{b},\xi }(f)$ (resp. $%
E_{\partial _{b},\xi }(f)$) will be referred to as a $\overline{\partial }%
_{b}$-harmonic map (resp. $\partial _{b}$-harmonic map). Clearly $E_{\overline{%
\partial }_{b},\xi }(f)=0$ (resp. $E_{\partial _{b},\xi }(f)=0$) if and only
if $f$ is a foliated CR map (resp. foliated anti-CR map). 

For a map $f:(M^{2m+1},H,J,\theta )\rightarrow (N,\widetilde{J},\widetilde{g}) $, we set
\[
K_{b}(f)=E_{\partial _{b},\xi }(f)-E_{\overline{\partial }_{b},\xi
}(f)=E_{\partial _{b}}(f)-E_{\overline{\partial }_{b}}(f).
\]%
The authors in \cite{SSZ2013holomorphic} proved that if $M$ is a compact Sasakian manifold,
then $K_{b}(f)$ is invariant under a foliated deformation. First, we want to
generalize their result to the case that the domain manifold is a general
pseudo-Hermitian manifold.

\begin{theorem}
 Let $(M^{2m+1},H,J,\theta )$ be a compact
pseudo-Hermitian manifold, and $(N,\widetilde{J},\widetilde{g})$ be a K\"{a}hler
manifold. Then $K_{b}(f)$ is a smooth foliated homotopy invariant, that is, $%
K_{b}(f_{t})$ is constant for any family $\{f_{t}\}$ of foliated maps. 
\end{theorem}

This is a foliated Lichnerowicz type result, which implies that the $E_{%
\overline{\partial }_{b},\xi }$-, $E_{\partial _{b},\xi }$- and $E$%
-critical points through foliated maps coincide. Furthermore, in a given
foliated homotopy class, the $E_{\overline{\partial }_{b},\xi }$-, $%
E_{\partial _{b},\xi }$- and $E$-minima coincide. 

Next, we try to generalize Petit's foliated rigidity theorem and get the
following result.

\begin{theorem}
  Let $(M^{2m+1},H,J,\theta )$ be a compact Sasakian
manifold with $m\geq 2$, and $(N,\widetilde{J},\widetilde{g})$ be a K\"{a}%
hler manifold with strongly semi-negative curvature. If $f:M\rightarrow N$
is a $\overline{\partial }_{b}$-harmonic map or a $\partial _{b}$-harmonic
map, then $f$ is foliated. Furthermore, $f$ must be $\overline{\partial }_{b}
$-pluriharmonic (that is, $f_{i\overline{j}}^{\alpha }=f_{\overline{j}%
i}^{\alpha }=0$), and 
\[
    \widetilde{R}_{ \beta \bar \alpha \gamma \bar \sigma } (f_{\bar i}^{\bar \alpha  }f_{\bar j}^{\beta }-f_{\bar j}^{\bar \alpha  }f_{\bar i}^{\beta }) (\overline{f_{\bar i}^{\bar \gamma }f_{ \bar j}^{ \sigma  } -f_{\bar j}^{\bar \gamma }f_{\bar i}^{\sigma  }}) =0.
\]
\end{theorem}

Subsequently, by a similar argument as in \cite{Siu1980rigid}, \cite{Jost1991nonlinear} and \cite{CDRY2019onharm}, we obtain the following CR rigidity result for $\dbb$-harmonic maps.

\begin{theorem}
 Let $(M^{2m+1},H,J,\theta )$ be a compact Sasakian
manifold with $m\geq 2$, and $(N,\widetilde{J},\widetilde{g})$ be a K\"{a}%
hler manifold with strongly negative curvature. Suppose $f:M\rightarrow N$
is a $\overline{\partial }_{b}$-harmonic map, and $\rank_{\rr}(df_{p})\geq 3$ at
some point $p\in M$. Then $f$ is a foliated CR map or foliated anti-CR map.
\end{theorem}

\section{Preliminaries}

Let $M^{2m+1}$ be a $(2m+1)$-dimensional smooth orientable manifold. A \emph{%
CR structure} on $M^{2m+1}$ is a complex rank-$m$ subbundle $H^{1,0} $ of $%
T(M)\otimes \mathbb{C}$ with the following properties 
\begin{align}
    H^{1,0} \cap H^{0,1}  =& \{0\}, \quad H^{0,1}  = \overline{H^{1,0} } \notag \\
    [\Gamma (H^{1,0} ),\Gamma (H^{1,0} )] \subseteq &\, \Gamma(H^{1,0} ). \label{1.1}
    \end{align}    
The complex subbundle $H^{1,0} $ corresponds to a real rank $2m$ subbundle 
$H :=\Re\{H^{1,0} \oplus H^{0,1} \}$ of $T(M)$, which carries a
complex structure $J_{b}$ defined by 
\begin{equation*}
J_{b}(V+\overline{V})=i(V-\overline{V})
\end{equation*}%
for any $V\in H^{1,0} $. The synthetic object $(M,H^{1,0} )$ or $%
(M,H ,J_{b})$ is called a \emph{CR manifold}.

Let $E$ be a real line bundle of $T^{\ast }M$, whose fiber at each point $%
x\in M$ is given by%
\begin{equation*}
E_{x}=\{\omega \in T_{x}^{\ast }M:\ker \omega \supseteq H_{x} \}.
\end{equation*}%
Since both $TM$ and $H $ are orientable vector bundles on $M$, the real
line bundle $E$ is orientable, so $E$ has globally defined nowhere vanishing
sections. Any such a section $\theta \in \Gamma (E\backslash \{0\})$ is
referred to as a \emph{pseudo-Hermitian 1-form} on $M$.

Given a pseudo-Hermitian $1$-form $\theta $ on $M$, we have the Levi form $%
L_{\theta }$ corresponding to $\theta $, which is defined by%
\begin{equation}
L_{\theta }(X,Y)=d\theta (X,J_{b}Y)  \label{1.2}
\end{equation}%
for any $X,Y\in H $. The second condition in (\ref{1.1}) implies that $%
L_{\theta }$ is $J_{b}$-invariant, and thus symmetric. If $L_{\theta }$ is
positive definite on $H $ for some $\theta $, $(M,H^{1,0} )$ is said to
be strictly pseudoconvex. From now on, we will always assume that $%
(M,H^{1,0} )$ is a strictly pseudoconvex CR manifold endowed with a
pseudo-Hermitian $1$-form $\theta $, such that its Levi form $L_{\theta }$
is positive definite. In this case  the synthetic object $(M,H^{1,0} ,J,\theta )$ is
referred to as a pseudo-Hermitian manifold.

Let $(M^{2m+1},H^{1,0} ,J, \theta )$ be a pseudo-Hermitian manifold. Clearly $%
\theta $ is a contact form. Thus there is a unique vector field $\xi \in
\Gamma (T(M))$, called the \emph{Reeb vector field}, such that
\begin{equation}
\theta (\xi )=1\text{, \ }i_{\xi }d\theta =0 , \label{1.3}
\end{equation}
where $i_{\xi }$ denotes the interior product with respect to $\xi $. The
collection of all its integral curves forms an oriented one-dimensional
foliation $\mathcal{F}_{\xi }$ on $M$, which is called the \emph{Reeb
foliation}. The first condition in (\ref{1.3}) implies that $\xi $ is
transversal to $H $. Therefore $T(M)$ admits a decomposition%
\begin{equation}
T(M)=H \oplus V_{\xi },  \label{1.4}
\end{equation}%
where $V_{\xi }:=%
\mathbb{R}
\xi $ is a trivial line bundle on $M$. In terms of terminology from
foliation theory, $H $ and $V_{\xi }$ are called the horizontal and
vertical distributions respectively. Let $\pi _{H}:TM\rightarrow H $ and $%
\pi _{V}:TM\rightarrow V_{\xi }$ be the natural projections associated with
the direct sum decomposition (\ref{1.4}). In terms of $\theta $, the Levi
form $L_{\theta }$ can be extended to a Riemannian metric%
\begin{equation}
g_{\theta }=L_{\theta }(\pi _{H},\pi _{H})+\theta \otimes \theta,  \label{1.5}
\end{equation}
which is called the Webster metric. It is convenient to extend the complex
structure $J_{b}$ on $H $ to an endomorphism $J$ of $T(M)$ by requiring
that 
\begin{equation}
J\mid _{H }=J_{b}\text{ \ and \ }J\mid _{V_{\xi }}=0 , \label{1.6}
\end{equation}%
where $\mid $ denotes the fiberwise restriction.

It is known that there exists a unique linear connection $\nabla $ on $%
(M^{2m+1},H^{1,0} ,\theta )$, called the \emph{Tanaka-Webster connection},
such that (cf. \cite{DT2006differential}, \cite{TA1975adifferential}, \cite{We1978pseudo})
\begin{enumerate}
\item $\nabla _{X}\Gamma (H )\subseteq \Gamma (H )$ and $\nabla _{X}J=0$
for any $X\in \Gamma (TM)$;

\item $\nabla g_{\theta }=0$;

\item $T_{\nabla }(X,Y)=2d\theta (X,Y)\xi $ \ and $T_{\nabla }(\xi
,JX)+JT_{\nabla }(\xi ,X)=0$ for any $X,Y\in H $, where $T_{\nabla }(\cdot
,\cdot )$ denotes the torsion of the connection $\nabla $.
\end{enumerate}

One important partial component of $T_{\nabla }$ is the pseudo-Hermitian
torsion $\tau $ given by%
\begin{equation}
\tau (X)=T_{\nabla }(\xi ,X)  \label{1.7}
\end{equation}%
for any $X\in TM$. Then $(M,H^{1,0} ,\theta )$ is said to be \emph{Sasakian%
} if $\tau =0$.

For the pseudo-Hermitian manifold $(M,H^{1,0} ,\theta )$, we choose a
local orthonormal frame field $\{e_{A}\}_{A=0}^{2m}=\{\xi
,e_{1},...,e_{m},e_{m+1},...,e_{2m}\}$ with respect to $g_{\theta }$ such
that 
\begin{equation*}
\{e_{m+1},...,e_{2m}\}=\{Je_{1},...,Je_{m}\}.
\end{equation*}%
Such a frame field $\{e_{A}\}_{A=0}^{2m}$ is referred to as an adapted frame
field $M$. Set 
\begin{equation}
\eta _{j}=\frac{1}{\sqrt{2}}\left( e_{j}-\sqrt{-1}Je_{j}\right) \text{, \ }%
\eta _{\bar{j}}=\overline{\eta _{j}}\text{ \ (}j=1,...,m\text{).}
\label{1.8}
\end{equation}%
Let $\{\theta ^{j}\}_{j =1}^{m}$ be the dual frame field of $\{\eta
_{j}\}_{j=1}^{m}$. By the properties of the Tanaka-Webster connection $%
\nabla $, we have (cf. \cite{DT2006differential})%
\begin{equation}
\nabla \xi =0\text{, \ }\nabla \eta _{j}=\theta _{j}^{i}\otimes \eta _{i}%
\text{, \ }\nabla \eta _{\overline{j}}=\theta _{\overline{j}}^{\overline{i}%
}\otimes \eta _{\overline{i}} , \label{1.9}
\end{equation}%
where $\{\theta _{j}^{i}\}$ denotes the connection $1$-forms with respect to
the frame field. Since $\tau (H^{1,0} )\subset H^{0,1} $, one may write 
\begin{eqnarray}
\tau &=&\tau ^{i}\eta _{i}+\tau ^{\overline{i}}\eta _{\overline{i}}  \notag
\\
&=&A_{\overline{j}}^{i}\theta ^{\overline{j}}\otimes \eta _{i}+A_{j}^{%
\overline{i}}\theta ^{j}\otimes \eta _{\overline{i}}.  \label{1.10}
\end{eqnarray}%
From \cite{We1978pseudo}, we know that $\{\theta ,\theta ^{i},\theta_{j}^{i}\}$
satisfy the following structure equations (cf. also \S 1.4 in \cite{DT2006differential}%
)%
\begin{eqnarray}
d\theta &=&2\sqrt{-1}\theta ^{i}\wedge \theta ^{\overline{i}}\text{,}  \notag
\\
d\theta ^{i} &=&-\theta _{j}^{i}\wedge \theta ^{j}+A_{\bar{j}%
}^{i}\theta \wedge \theta ^{\overline{j}}\text{,}  \label{1.11} \\
d\theta _{j}^{i} &=&-\theta _{k}^{i}\wedge \theta _{j}^{k}+\Pi _{j}^{i} 
\notag
\end{eqnarray}%
with
\begin{eqnarray}
\Pi _{j}^{i} &=&2\sqrt{-1}(\theta ^{i}\wedge \tau ^{\overline{j}}-\tau
^{i}\wedge \theta ^{\overline{j}})+R_{jk\bar{l}}^{i}\theta ^{k}\wedge
\theta ^{\overline{l}}  \notag \\
&&+W_{j\overline{k}}^{i}\theta \wedge \theta ^{\overline{k}%
}-W_{jk}^{i}\theta \wedge \theta ^{k} , \label{1.12}
\end{eqnarray}%
where $W_{j\overline{k}}^{i}=A_{\overline{k},j}^{i}$, $W_{jk}^{i}=A_{j,%
\overline{i}}^{\overline{k}}$ are the covariant derivatives of $A$ and $R_{jk%
\bar{l}}^{i}$ are the components of curvature tensor of the
Tanaka-Webster connection.

\begin{lemma}[{\protect\cite{CDRY2019onharm}}]\label{lemDiv}
Let $(M^{2m+1},H,J,\theta )$ be a pseudo-Hermitian manifold with
Tanaka-Webster connection $\nabla $. Let $X$ and $\rho $ be a vector field
and $1$-form on $M$ respectively. Then%
\begin{equation*}
\di X=\sum_{A=0}^{2m}g_{\theta }(\nabla _{e_{A}}X,e_{A})\text{ \ and \ 
}\delta \rho =-\sum_{A=0}^{2m}(\nabla _{e_{A}}\rho )(e_{A}),
\end{equation*}%
where $\{e_{A}\}_{A=0}^{2m}=\{\xi ,e_{1},...,e_{2m}\}$ is an orthonormal
frame field on $M$. Here $\di(\cdot )$ and $\delta (\cdot )$ denote
the divergence and codifferential respectively.
\end{lemma}

\begin{definition}
A map $f:(M,H,J)\rightarrow (N,\widetilde{J})$ from a CR manifold to a
complex manifold is called a \emph{CR map }(resp. anti CR map) if $%
df(H^{1,0} )\subset T^{1,0}(N)$ (resp. $df(H^{0,1} )\subset T^{1,0}(N)$%
), equivalently, $df_{H}\circ J=\widetilde{J}\circ df_{H}$ (resp. $%
df_{H}\circ J=-\widetilde{J}\circ df_{H}$), where $df_{H}=df\mid _{H}$. In
particular, if $N=\mathbb{C}$, then $f$ is called a \emph{CR function }%
(resp.  anti CR function).
\end{definition}

A map $f:(M,H,J,\theta )\rightarrow N$ from a pseudo-Hermitian manifold to a
smooth manifold is said to be \emph{foliated} if $df(\xi )=0$. Here the target manifold is regarded as a trivial foliation by points.
In \cite{CDRY2019onharm} and \cite{GIP2001cr}, the following type of generalized
holomorphic maps was investigated.

\begin{definition}[\cite{GIP2001cr}]
A smooth map $f:(M,H,J,\theta )\rightarrow (N,\widetilde{J})$ from a
pseudo-Hermitian manifold to a complex manifold is called $(J,\widetilde{J})$%
-holomorphic (resp. anti $(J,\widetilde{J})$-holomorphic) if it satisfies $%
df\circ J=\widetilde{J}\circ df$ (resp. $df\circ J=-\widetilde{J}\circ df$).
\end{definition}

\begin{remark}
Clearly $f:(M,H,J,\theta )\rightarrow (N,\widetilde{J})$ is a $(J,\widetilde{%
J})$-holomorphic map if and only if it is a foliated CR map. Note that $(J,\widetilde{J})$-holomorphic map  is also called CR-holomorphic map in \cite{Pet2002harmonic}. 
\end{remark}

Let $f:(M^{2m+1},H,J,\theta )\rightarrow (N,\widetilde{J},\widetilde{g})$ be
a map from a pseudo-Hermitian manifold to a K\"{a}hler manifold. We have the
partial differentials 
\begin{equation*}
\overline{\partial }_{b}f:H^{0,1}\rightarrow T^{1,0}N\text{, \ }\partial
_{b}f:H^{1,0}\rightarrow T^{1,0}N
\end{equation*}%
defined by 
\begin{equation*}
\overline{\partial }_{b}f=\pi ^{1,0}(df\mid _{H^{0,1}})\text{, \ }\partial
_{b}f=\pi ^{1,0}(df\mid _{H^{1,0}}),
\end{equation*}%
where $\pi ^{1,0}:T^{\cc}N\rightarrow T^{1,0}N$ is the natural projection
morphism. Let $\{e_{0},e_{1},...,e_{2m}\}$ be the adapted frame field on $M$
as given above. Similarly, let $\{\widetilde{e}_{1},...,\widetilde{e}_{2n}\}$
be a local orthonormal frame field on $(N,\widetilde{J},\widetilde{g})$ with 
$\widetilde{e}_{n+1}=\widetilde{J}\widetilde{e}_{1},...,\widetilde{e}_{2n}=\widetilde{J}\widetilde{e}_{n}$. Set%
\begin{equation}
\widetilde{\eta }_{\alpha }=\frac{1}{\sqrt{2}}(\widetilde{e}_{\alpha }-\sqrt{-1}
\widetilde{J}\widetilde{e}_{\alpha })\text{, \ \ }\alpha =1,...,n.
\label{1.13}
\end{equation}%
Let $\{\widetilde{\theta }^{\alpha }\}_{\alpha =1}^{n}$ be the dual frame
field of $\{\widetilde{\eta }_{\alpha }\}_{\alpha =1}^{n}$. In terms of the frame
fields, we can write 
\begin{equation}
\overline{\partial }_{b}f=f_{\overline{j}}^{\alpha }\theta ^{\overline{j}%
}\otimes \widetilde{\eta }_{\alpha }\text{, \ }\partial _{b}f=f_{j}^{\alpha
}\theta ^{j}\otimes \widetilde{\eta }_{\alpha }.  \label{1.14}
\end{equation}%
Then%
\begin{equation}
\mid \overline{\partial }_{b}f\mid ^{2}=\sum_{j,\alpha }f_{\overline{j}%
}^{\alpha }f_{j}^{\overline{\alpha }}\text{, \ }\mid \partial _{b}f\mid
^{2}=\sum_{j,\alpha }f_{j}^{\alpha }f_{\overline{j}}^{\overline{\alpha }},
\label{1.15}
\end{equation}%
or%
\begin{equation}
\begin{array}{c}
\mid \overline{\partial }_{b}f\mid ^{2}=\frac{1}{4}\{\langle
df(e_{j}),df(e_{j})\rangle +\langle df(Je_{j}),df(Je_{j})\rangle \\ 
-2\langle df(Je_{j}),\widetilde{J}df(e_{j})\rangle \} \\ 
=\frac{1}{4}\sum_{A=1}^{2m}\left\{ \langle df(e_{A}),df(e_{A})\rangle
-\langle \widetilde{J}df(e_{A}),df(Je_{A})\rangle \right\},%
\end{array}
\label{1.16}
\end{equation}%
\begin{equation}
\begin{array}{c}
\mid \partial _{b}f\mid ^{2}=\frac{1}{4}\{\langle df(e_{j}),df(e_{j})\rangle
+\langle df(Je_{j}),df(Je_{j})\rangle \\ 
+2\langle df(Je_{j}),\widetilde{J}df(e_{j})\rangle \} \\ 
=\frac{1}{4}\sum_{A=1}^{2m}\left\{ \langle df(e_{A}),df(e_{A})\rangle
+\langle \widetilde{J}df(e_{A}),df(Je_{A})\rangle \right\}.%
\end{array}
\label{1.17}
\end{equation}%
Then we can introduce the following two energy functionals%
\begin{equation}\label{energyfunc1}
E_{\overline{\partial }_{b},\xi }(f)=\int_{M}\left\{ \mid \overline{\partial 
}_{b}f\mid ^{2}+\frac{1}{4}\mid df(\xi )\mid ^{2}\right\} dv_{\theta },
\end{equation}%
and%
\begin{equation}\label{energyfunc2}
E_{\partial _{b},\xi }(f)=\int_{M}\left\{ \mid \partial _{b}f\mid ^{2}+\frac{%
1}{4}\mid df(\xi )\mid ^{2}\right\} dv_{\theta }  ,
\end{equation}%
where $\xi $ is the Reeb vector field of $(M,\theta )$. Clearly $E_{%
\overline{\partial }_{b},\xi }(f)\equiv 0$ (resp. $E_{\partial _{b},\xi
}(f)\equiv 0$) if and only if $f$ is a foliated CR map (resp. foliated anti-CR
map).

\begin{definition}
A critical point of $E_{\overline{\partial }_{b},\xi }(f)$ (resp. $%
E_{\partial _{b},\xi }(f)$) is called a $\overline{\partial }_{b}$-harmonic
map (resp. $\partial _{b}$-harmonic map).
\end{definition}
\begin{remark}
In \cite{LS2019CRanalogue}, Li and Son  introduced the \emph{ $\dbb $-energy functional} $E_{\dbb }(f)$ of $f$. Compared to their definition, we include the term $\frac{1}{4}|df(\xi )|^2$ in \eqref{energyfunc1}.
\end{remark}

For a map $f:(M,H^{1,0} ,\theta )\rightarrow (N,\widetilde{J},\widetilde{g}%
)$, we define its second fundamental form by%
\begin{equation*}
\beta (X,Y)=\widetilde{\nabla }_{Y}df(X)-df(\nabla _{Y}X)
\end{equation*}%
for any $X,Y\in \Gamma (TM)$, where $\nabla $ and $\widetilde{\nabla }$
denote the Tanaka-Webster connection of $M$ and the Levi-Civita connection
of $N$, respectively. The notion of the above second fundamental form has
appeared in literature in various special cases (cf. \cite{EL1983selected}, \cite{DK2010pseudo}, \cite{Pet2002harmonic}, \cite{Pet2009mok}, etc.).

\begin{lemma}
(cf. \cite{donghhharmonicMapsPseudoHermitian2016a}) Let $f:(M,\nabla )\rightarrow (N,\widetilde{\nabla })$
be a map between manifolds with the linear connections. Then%
\begin{equation*}
\widetilde{\nabla }_{X}df(Y)-\widetilde{\nabla }_{Y}df(X)-df([X,Y])=T_{%
\widetilde{\nabla }}(df(X),df(Y))
\end{equation*}%
for any $X,Y\in \Gamma (TM)$, where $T_{\widetilde{\nabla }}$ denotes the
torsion of $\widetilde{\nabla }$. Equivalently, we have%
\begin{equation*}
\beta (X,Y)-\beta (Y,X)=df(T_{\nabla }(X,Y))-T_{\widetilde{\nabla }%
}(df(X),df(Y)).
\end{equation*}
\end{lemma}

Now we want to derive the variation formulas of the energy functionals
defined by (\ref{energyfunc1}) and \eqref{energyfunc2}.

\begin{lemma}\label{lem:EL}
Let $(M^{2m+1},H,J,\theta )$ be a pseudo-Hermitian manifold and $(N,%
\widetilde{J},\widetilde{g})$ be a K\"{a}hler manifold. Suppose $%
\{f_{t}\}_{\mid t\mid <\varepsilon }$ is a family of maps from $M\ $to $N$
with $f_{0}=f$ and $v=(\partial f_{t}/\partial t)\mid _{t=0}\in \Gamma
(f^{-1}TN)$. Then 
\begin{equation*}
\frac{dE_{\overline{\partial }_{b},\xi }(f_{t})}{dt}\mid _{t=0}=-\frac{1}{2}%
\int_{M}\langle v,\trace_{g_{\theta }}\beta -2m\widetilde{J}df(\xi )\rangle
\end{equation*}%
and 
\begin{equation*}
\frac{dE_{\partial _{b},\xi }(f_{t})}{dt}\mid _{t=0}=-\frac{1}{2}%
\int_{M}\langle v,\trace_{g_{\theta }}\beta +2m\widetilde{J}df(\xi )\rangle .
\end{equation*}
\end{lemma}

\begin{proof}
Set $F:M\times (-\varepsilon ,\varepsilon )\rightarrow N$ by $%
F(x,t)=f_{t}(x) $ for any $x\in M$ and $t\in (-\varepsilon ,\varepsilon )$.
Then%
\begin{eqnarray}
\frac{dE_{\overline{\partial }_{b},\xi }(f_{t})}{dt} &\mid &_{t=0}  \notag \\
&=&\frac{1}{4}\int_{M}\sum_{A=1}^{2m}\{2\langle \widetilde{\nabla }_{\frac{%
\partial }{\partial t}}dF(e_{A}),dF(e_{A})\rangle -\langle \widetilde{J}%
\widetilde{\nabla }_{\frac{\partial }{\partial t}}dF(e_{A}),dF(Je_{A})\rangle
\notag \\
&&-\langle \widetilde{J}dF(e_{A}),\widetilde{\nabla }_{\frac{\partial }{%
\partial t}}dF(Je_{A})\rangle \}dv_{\theta }\text{ }+\frac{1}{2}%
\int_{M}\langle \widetilde{\nabla }_{\frac{\partial }{\partial t}}dF(\xi
),dF(\xi )\rangle dv_{\theta }  \notag \\
&=&\sum_{A=1}^{2m}\int_{M}\frac{1}{4}\{2\langle \widetilde{\nabla }%
_{e_{A}}v,df(e_{A})\rangle -\langle \widetilde{J}\widetilde{\nabla }%
_{e_{A}}v,df(Je_{A})\rangle  \notag \\
&&-\langle \widetilde{J}df(e_{A}),\widetilde{\nabla }_{Je_{A}}v\rangle
\}dv_{\theta }+\frac{1}{2}\int_{M}\langle \widetilde{\nabla }_{\xi }v,df(\xi
)\rangle dv_{\theta }  \label{1.20} \\
&=&\frac{1}{2}\sum_{A=0}^{2m}\int_{M}\{\langle \widetilde{\nabla }%
_{e_{A}}v,df(e_{A})\rangle +\langle \widetilde{\nabla }_{e_{A}}v,\widetilde{J%
}df(Je_{A})\rangle \}dv_{\theta }  \notag \\
&=&\frac{1}{2}\sum_{A=0}^{2m}\int_{M}\{e_{A}\langle v,df(e_{A})\rangle
-\langle v,df(\nabla _{e_{A}}e_{A})\rangle -\langle v,(\widetilde{\nabla }%
_{e_{A}}df)(e_{A})\rangle  \notag \\
&&+e_{A}\langle v,\widetilde{J}df(Je_{A})\rangle -\langle v,\widetilde{J}%
dfJ(\nabla _{e_{A}}e_{A})\rangle -\langle v,(\widetilde{\nabla }_{e_{A}}%
\widetilde{J}dfJ)(e_{A})\rangle \} . \notag
\end{eqnarray}%
Define a $1$-form $\rho $ by $\rho (X)=\langle v,df(X)\rangle +\langle v,%
\widetilde{J}\circ df\circ J(X)\rangle $ for any $X\in TM$. By Lemma \ref{lemDiv}, we
deduce that%
\begin{equation}
\delta \rho =-\sum_{A=0}^{2m}(\nabla _{e_{A}}\rho )(e_{A}).  \label{1.21}
\end{equation}%
It follows from (\ref{1.20}) and (\ref{1.21}) that 
\begin{equation}
\frac{dE_{\overline{\partial }_{b},\xi }(f_{t})}{dt}\mid _{t=0}=-\frac{1}{2}%
\int_{M}\langle v,\sum_{A=0}^{2m}(\widetilde{\nabla }_{e_{A}}df)(e_{A})+[%
\widetilde{\nabla }_{e_{A}}(\widetilde{J}\circ df\circ J)](e_{A})\rangle .
\label{1.22}
\end{equation}%
Next, 
\begin{eqnarray*}
\sum_{A=1}^{2m}[\widetilde{\nabla }_{e_{A}}(\widetilde{J}\circ df\circ
J)](e_{A}) &=&\sum_{A=1}^{2m}\widetilde{\nabla }_{e_{A}}(\widetilde{J}\circ
df\circ Je_{A})-\widetilde{J}\circ df\circ J(\nabla _{e_{A}}e_{A}) \\
&=&\sum_{A=1}^{2m}\widetilde{J}\left[ \widetilde{\nabla }%
_{e_{A}}df(Je_{A})-df(\nabla _{e_{A}}Je_{A})\right] \\
&=&\sum_{A=1}^{2m}\widetilde{J}\beta (Je_{A},e_{A}) \\
&=&\sum_{j=1}^{m}\widetilde{J}\left[ \beta (Je_{j},e_{j})-\beta
(e_{j},Je_{j})\right] \\
&=&\sum_{j=1}^{m}\widetilde{J}df(T_{\nabla }(Je_{j},e_{j})) \\
&=&-2m\widetilde{J}df(\xi ).
\end{eqnarray*}%
Then we get from (\ref{1.22}) the variation formula for $E_{\overline{%
\partial }_{b},\xi }(f)$. The variation formula for $E_{\partial _{b},\xi
}(f)$ may be derived in a similar way. Hence we complete the proof of this
lemma.
\end{proof}
 Define the tension field $\tau _{\dbb,\xi }(f) $ of  $f $ with respect to the functional $E_{\dbb,\xi } $ by 
$$ \tau _{\dbb,\xi }(f):= \trace_{g_{\theta }}\beta -2m\widetilde{J}df(\xi ).$$
Then, according to  Lemma \ref{lem:EL}, $f$ is $\dbb$-harmonic  if and only if $\tau _{\dbb, \xi }(f)=0$.

Note that $\tau _{\overline{\partial }_{b},\xi }(f)=0$ (or $\tau _{\partial
_{b},\xi }(f)=0$) is a system of elliptic differential equations that
differs from the harmonic map equation by a linear first-order term. By a
similar argument as in \cite{Samp1978some}, we have that

\begin{theorem}[Unique continuation]
    Let $f :(M^{2m+1},H,J, \theta ) \to (N^{2n},\widetilde{J},\widetilde{g})$ be a $\dbb $-harmonic map or $\db $-harmonic map. If $f$ is  constant on a non-empty open subset $U$ of $M$, then $f $ is constant on $M$. 
\end{theorem}

Let us recall some  definitions of   generalized  harmonic maps from pseudo-Hermitian manifolds.
\begin{definition}
Let $ (M^{2m+1},H,J, \theta )$ be a pseudo-Hermitian manifold and $ (N^{2n},\widetilde{J},\widetilde{g}) $ be a K\"ahler manifold. Suppose
$f :M \to N$ is a smooth map. We say $f $ is 
\begin{enumerate}[(i)]
    \item (\cite{BDU2001pseudoh}) pseudoharmonic, if $ \trace_{g_\theta }(\pi_{H}\beta )=0 $;
    \item  (\cite{LS2019CRanalogue}) pseudo-Hermitian harmonic, if it is a critical point of $E_{\dbb}(\cdot)$;
    \item (\cite{DK2010pseudo}) $\dbb $-pluriharmonic, if $\beta (X,Y)+\beta (JX,JY)=0$ for all $X,Y \in H $.
\end{enumerate}
\end{definition}
\begin{remark}
Clearly, we have the following
\begin{enumerate}[(a)]
    \item If $f $ is $\dbb $-pluriharmonic, then it must be pseudoharmonic (\cite{dragomirPseudoharmonicMapsVector2010a});
    \item If $f$ is a CR map, then $f$ is pseudo-Hermitian harmonic;
    \item If $f $ is a CR map (resp. anti-CR map), then $f $ is $\dbb $-harmonic (resp. $\db$-harmonic) if and only if $\beta (\xi ,\xi )=0 $ (cf. \eqref{Lf});
    \item If $f $ is foliated, then notions of  $\dbb $-harmonic,  $\db $-harmonic, pseudoharmonic, pseudo-Hermitian harmonic and harmonic maps coincide. 
\end{enumerate}
Besides, as proved in \cite{CDRY2019onharm}, if $f $ is $\dbb $-pluriharmonic, then it is foliated; if $f $ is  $\pm (J, \widetilde{J})$-holomorphic, then it is $\dbb $-pluriharmonic. 
\end{remark}

\section{Lichnerowicz type results}

In this section, we generalize the Lichnerowicz type result in \cite{SSZ2013holomorphic} to the case that the domain manifold is a general
pseudo-Hermitian manifold.

Let $f:(M^{2m+1},H,J,\theta )\rightarrow (N,\widetilde{J},\omega  ^{N})$ be a
smooth map from a pseudo-Hermitian manifold to a K\"{a}hler manifold, where $%
\omega  ^{N}$ is the K\"{a}hler form of $N$, given by $\omega  ^N(X,Y)=\widetilde{g}(JX,Y)$ for all $X,Y \in TN$. 
Set 
\begin{equation}
k_{b}(f)=\mid \partial _{b}f\mid ^{2}-\mid \overline{\partial }_{b}f\mid ^{2}
\label{L1}
\end{equation}%
and 
\begin{equation}
K_{b}(f)=E_{\partial _{b},\xi }(f)-E_{\overline{\partial }_{b},\xi }(f).
\label{L2}
\end{equation}

\begin{lemma}\label{lem3.1}
Under the above notations, we have%
\begin{equation*}
k_{b}(f)=\langle d\theta ,f^{\ast }\omega ^{N}\rangle .
\end{equation*}
\end{lemma}

\begin{proof}
Let $\{\xi ,e_{1},...,e_{m},Je_{1},...,Je_{m}\}$ be an adapted frame on $M$.
Using (\ref{1.2}), (\ref{1.16}) and (\ref{1.17}), we deduce that%
\begin{eqnarray*}
\langle d\theta ,f^{\ast }\omega ^{N}\rangle &=&\sum_{i<j}\left\{ (f^{\ast
}\omega ^{N})(e_{i},e_{j})d\theta (e_{i},e_{j})+(f^{\ast }\omega
^{N})(Je_{i},Je_{j})d\theta (Je_{i},Je_{j})\right\} \\
&&+\sum_{i,j}(f^{\ast }\omega ^{N})(e_{i},Je_{j})d\theta (e_{i},Je_{j}) \\
&=&\sum_{i}\langle \widetilde{J}df(e_{i}),df(Je_{i})\rangle \\
&=&k_{b}(f).
\end{eqnarray*}
\end{proof}

The following lemma is useful.

\begin{lemma}[Homotopy Lemma, cf. \cite{Lich1970app}, \cite{EL1983selected}]\label{lem3.2}
Let $f_{t}:M\rightarrow N$ be a family of smooth maps between smooth
manifolds, parameterized by real number $t$, and let $\omega $ be a closed
two-form on $N$. Then%
\begin{equation*}
\frac{\partial }{\partial t}\left( f_{t}^{\ast }\omega \right) =d\left(
f_{t}^{\ast }i(\frac{\partial f_{t}}{\partial t})\omega \right),
\end{equation*}%
where the notation $i(X)$ denotes the interior product with respect to the
vector $X$.
\end{lemma}

\begin{lemma}
Let $f_{t}:(M^{2m+1},H,J,\theta )\rightarrow (N,\widetilde{J},\omega ^{N})$
be a family of smooth maps from a compact pseudo-Hermitian manifold to a K%
\"{a}hler manifold. Then%
\begin{equation*}
\frac{d}{dt}K_{b}(f_{t})=2m\int_{M}\omega ^{N}(v_{t},df_{t}(\xi ))dv_{\theta
},
\end{equation*}%
where $v_{t}=\partial f_{t}/\partial t$.
\end{lemma}

\begin{proof}
In terms of Lemma \ref{lem3.1} and Lemma \ref{lem3.2}, we have
\begin{eqnarray*}
\frac{d}{dt}K_{b}(f_{t}) &=&\int_{M}\langle \frac{\partial }{\partial t}%
f_{t}^{\ast }\omega ^{N},d\theta \rangle dv_{\theta } \\
&=&\int_{M}\langle d\left( f_{t}^{\ast }i(\frac{\partial f_{t}}{\partial t}%
)\omega ^{N}\right) ,d\theta \rangle dv_{\theta } \\
&=&\int_{M}\langle f_{t}^{\ast }i(\frac{\partial f_{t}}{\partial t})\omega
^{N},\delta d\theta \rangle dv_{\theta }.
\end{eqnarray*}
Recall that (cf. \cite{DT2006differential})
\begin{equation*}
\nabla _{X}^{\theta }Y=\nabla _{X}Y-(d\theta (X,Y)+A(X,Y))\xi +\theta
(Y)\tau (X)+\theta (X)JY+\theta (Y)JX
\end{equation*}
for any $X,Y\in \Gamma (TM)$, where $\nabla ^{\theta }$ denotes the Levi-Civita connection of $g_{\theta }$.
  Let $\{e_{A}\}_{A=0}^{2m}=\{\xi
,e_{1},...,e_{2m}\}$ be an adapted frame field in $M$. For $X\in HM$, we
compute 
\begin{eqnarray*}
(\delta d\theta )(X) &=&-\sum_{A=0}^{2m}(\nabla _{e_{A}}^{\theta }d\theta
)(e_{A},X) \\
&=&-\sum_{A=0}^{2m}\{e_{A}d\theta (e_{A},X)-d\theta (\nabla _{e_{A}}^{\theta
}e_{A},X)-d\theta (e_{A},\nabla _{e_{A}}^{\theta }X)\} \\
&=&-\sum_{A=1}^{2m}\{e_{A}d\theta (e_{A},X)-d\theta (\nabla
_{e_{A}}e_{A},X)-d\theta (e_{A},\nabla _{e_{A}}X)\} \\
&=&-\sum_{A=1}^{2m}\left( \nabla _{e_{A}}d\theta \right) (e_{A},X) \\
&=&0,
\end{eqnarray*}
where the last equality is due to $\nabla d\theta =0$. Next,
\begin{eqnarray*}
\left( \delta d\theta \right) (\xi ) &=&\sum_{A=1}^{2m}d\theta (e_{A},\nabla
_{e_{A}}^{\theta }\xi ) \\
&=&\sum_{A=1}^{2m}d\theta (e_{A},\tau (e_{A})+Je_{A}) \\
&=&2m,
\end{eqnarray*}%
since
\begin{eqnarray*}
&&d\theta (e_{i},\tau (e_{i}))+d\theta (Je_{i},\tau Je_{i}) \\
&=&d\theta (e_{i},\tau (e_{i}))-d\theta (e_{i},\tau (e_{i})) \\
&=&0\text{.}
\end{eqnarray*}%
Therefore,
\begin{eqnarray*}
\frac{d}{dt}K_{b}(f_{t}) &=&\int_{M}\langle f_{t}^{\ast }i(\frac{\partial
f_{t}}{\partial t})\omega ^{N},\delta d\theta \rangle dv_{\theta } \\
&=&\int_{M}\langle f_{t}^{\ast }[\omega ^{N}(v_{t},\cdot )],\delta d\theta
\rangle dv_{\theta } \\
&=&\int_{M}\omega ^{N}(v_{t},df_{t}(\xi ))\delta d\theta (\xi
)dv_{\theta } \\
&=&2m\int_{M}\omega ^{N}(v_{t},df_{t}(\xi ))dv_{\theta }.
\end{eqnarray*}
\end{proof}

\begin{corollary}
 Let $f_{t}:(M^{2m+1},H,J,\theta )\rightarrow (N,\widetilde{J},\omega ^{N})$
be a family of smooth maps from a compact pseudo-Hermitian manifold to a K\"{a}hler manifold, such that $df_t(\xi )=0$ for every $t$. We refer to such $\{f_t\}$ as a family of foliated maps. Then $K_b(f_t)$ is a constant.
\end{corollary}

Thus, if $f_t:M \to N$ is a family of foliated maps, then 
$$ \frac{d}{dt} E_{\dbb,\xi }(f_t)=\frac{d}{dt} E_{\db,\xi }(f_t)=\frac{1}{2} \frac{d}{dt} E(f_t), $$
where $E(f)=E_{\dbb,\xi }(f)+E_{\db,\xi }(f)$ is the usual energy functional of $f$. Then, the following theorems are evident.

\begin{theorem}
    ~
\begin{enumerate}[(i)]
    \item The  $E_{\dbb,\xi }$-, $E_{\db,\xi }$- and $E$-critical points through foliated maps coincide. Moreover,  in a  given foliated homotopy class the $E_{\dbb,\xi }$-, $E_{\db,\xi }$- and $E$-minima coincide.
    \item If $f $ is $\pm (J, \widetilde{J})$-holomorphic, then it is  an absolute minimum of  $E$  in its foliated class.
\end{enumerate} 
\end{theorem}

\begin{proof}
   (i) For any $f, f_0$ in the same foliated homotopy class, the following equality holds:
    $$E_{\dbb,\xi }(f)-E_{\dbb,\xi }(f_0)=E_{\db,\xi }(f)-E_{\db,\xi }(f_0) .$$
   Consequently, if $E_{\dbb,\xi }(f_0)\leq E_{\dbb,\xi }(f)  $ for all $f$, 
then $E_{\db,\xi }(f_0)\leq E_{\db,\xi }(f)  $ for all $f$. Similarly, from the equality 
 $$E(f)-E(f_0)=2E_{\dbb,\xi }(f)-2E_{\dbb,\xi }(f_0),$$
  we conclude that $E_{\dbb,\xi }$ and $E$-minima coincide.

   (ii) A $ (J, \widetilde{J})$-holomorphic map (resp. anti-$ (J, \widetilde{J})$-holomorphic map) satisfies $E_{\dbb,\xi }(f)=0$ (resp. $E_{\db,\xi }(f)=0$) and is therefore an absolute minimum of $E$ in its foliated class.
\end{proof}

\begin{theorem}
    Let $f_{t}:(M^{2m+1},H,J,\theta )\rightarrow (N,\widetilde{J},\omega ^{N})$ be a family of foliated maps from a pseudo-Hermitian  manifold to a K\"ahler manifold with $0 \leq t \leq 1$. Suppose $f_0$ is $ (J, \widetilde{J})$-holomorphic and $f_1$ is anti-$ (J, \widetilde{J})$-holomorphic, then $f_0$ and $f_1$ are constant. In  particular, any   $\pm (J, \widetilde{J})$-holomorphic map in a trivial foliated homotopy class is constant. 
\end{theorem}
\begin{proof}
    Since $E_{\dbb,\xi }(f_0)=E_{\db,\xi }(f_1)=0$, $0 \leq E_{\db,\xi }(f_0)=-E_{\dbb,\xi }(f_1) \leq 0$, which leads to $ E_{\db,\xi }(f_0)=E_{\dbb,\xi }(f_1) =0$. Thus, $E(f_0)=E(f_1)=0$.
\end{proof}

\section{Commutation relations}\label{se:commut}

 In this section, we derive the commutation relations for maps from a pseudo-Hermitian manifold to a K\"ahler manifold.  While  the case of a map from a pseudo-Hermitian manifold to a general Riemannian manifold has been addressed in \cite{CDRY2019onharm}, we present it here using our notation for the sake of clarity and convenience.

 Let $f :(M^{2m+1},H,J, \theta ) \to (N^{2n},\widetilde{J},\widetilde{g})$ be a smooth map,  where $(M^{2m+1},H,J, \theta ) $ is a pseudo-Hermitian manifold and $(N^{2n},\widetilde{J},\widetilde{g})$ is a K\"ahler manifold.
Take $\{\theta^{i}\}$ as a local adapted coframe on $M$, and $\{\widetilde{\omega }^{\alpha }\}$ as a local orthonormal coframe on $N$ as aforementioned.
Unless otherwise stated, we adhere to the following index conventions:
$$
\begin{aligned}
&  A, B, C,D=0,1,\dots,m, \bar{1},\dots,\bar{m} ; \\
&  i, j, k,l,s =1,\dots, m;\\
& I, J, K,L,P=1,\dots,n, \bar{1},\dots,\bar{n} ; \\
&  \alpha ,\beta ,\gamma ,\sigma =1,\ldots   n,
\end{aligned}
$$
and employ the summation convention on repeated indices.
The structure equations for Levi-Civita connection $\widetilde{\nabla } $ on $(N, \widetilde{J}) $ can be expressed by 
$$
\begin{gathered}
d \widetilde{\omega}^\alpha=-\widetilde{\omega}_\beta^\alpha \wedge \widetilde{\omega}^\beta, \quad \widetilde{\omega}_\beta^\alpha+\widetilde{\omega}_{\bar{\alpha}}^{\bar{\beta}}=0, \\
d \widetilde{\omega}_\beta^\alpha=-\widetilde{\omega}_\gamma^\alpha \wedge \widetilde{\omega}_\beta^\gamma+\widetilde{\Omega}_\beta^\alpha,
\end{gathered}
$$
where $\widetilde{\Omega}_\beta^\alpha=\widetilde{R}_{\beta \gamma \bar{\sigma }}^\alpha \widetilde{\omega}^\gamma \wedge \widetilde{\omega}^{\bar{\sigma }}$. 
 Since $N $ is K\"ahler, the only possibly non-zero components of $ \widetilde{R}_{IJK}^{L} $ are 
 $$ \widetilde{R}_{\beta \gamma \bar{\sigma }}^\alpha ,\quad \widetilde{R}_{\bar \beta \gamma \bar{\sigma }}^{\bar \alpha} ,\quad \widetilde{R}_{\beta \bar \gamma \sigma }^\alpha ,\quad \widetilde{R}_{\bar \beta \bar \gamma \sigma }^{\bar \alpha }. $$
 Set 
 $$ \widetilde{R}_{IJKL}=\widetilde{g}( \widetilde{R}(\tilde{\eta }_{K}, \tilde{\eta }_{L})\tilde{\eta }_{J},\tilde{\eta }_{I} )=\widetilde{g}_{PI}  \widetilde{R}_{JKL }^{P }.$$
 Let
\begin{equation}
    \begin{aligned}
    d f & =f_A^{I} \theta^A \otimes \widetilde{\eta } _{I}, \\
    \beta  & =f_{A B}^{I} \theta^A \otimes \theta^B \otimes \widetilde{\eta }_{I}, \\
    \widetilde{\nabla} \beta  & =f_{A B C}^{I} \theta^A \otimes \theta^B \otimes \theta^C \otimes \widetilde{\eta } _{I},
    \end{aligned}
\end{equation}
where  $\widetilde{\nabla} \beta $ is the covariant derivative of $\beta $ with respect to $(\nabla,\widetilde{\nabla } )$, and recall that $\beta $ denotes the second fundamental form of $f$.
Thus we have
\begin{equation}\label{eq:1oder}
     f^* \widetilde{\omega }^{\alpha}=f_j^{\alpha} \theta^j+f_{\bar{j}}^{\alpha} \theta^{\bar{j}}+f_0^{\alpha} \theta . 
\end{equation}

Differentiating $ \eqref{eq:1oder}$, we have 
$$ \begin{aligned}
    f^* d \widetilde{\omega }^{\alpha}=&f_j^{\alpha} d\theta^j+f_{\bar{j}}^{\alpha} d\theta^{\bar{j}}+f_0^{\alpha} d\theta \\
    &+df_j^{\alpha} \wedge \theta^j+d f_{\bar{j}}^{\alpha} \wedge  \theta^{\bar{j}}+df_0^{\alpha}\wedge  \theta .\\
\end{aligned}  $$ 
By structure equations on $M $ and $N $, we have
$$ \begin{aligned}
    -f^* \widetilde{\omega}_{\beta}^{\alpha} \wedge f^* \widetilde{\omega}^{\beta}=& -f^* \widetilde{\omega}_{\beta}^{\alpha}  \wedge (f_j^{\beta} \theta^j+f_{\bar{j}}^{\beta} \theta^{\bar{j}}+f_0^{\beta} \theta )\\
    =&f_j^{\alpha} (\theta^k\wedge\theta_k^j+\theta\wedge\tau^j)+f_{\bar{j}}^{\alpha} (\theta^{\bar k}\wedge\theta_{\bar k}^{\bar j}+\theta\wedge\tau^{\bar j})+f_0^{\alpha} (2\sqrt{-1}h_{j\bar{k}}\theta^{j}\wedge\theta^{\bar k}) \\
    &+df_j^{\alpha} \wedge \theta^j+d f_{\bar{j}}^{\alpha} \wedge  \theta^{\bar{j}}+df_0^{\alpha}\wedge  \theta .\\
\end{aligned}  $$ 
After rearranging the above formula, we get 
\begin{equation}\label{eq:Df}
    D f_B^\alpha \wedge \theta^B+2 \sqrt{-1} f_0^\alpha h_{k \bar{l}} \theta^k \wedge \theta^{\bar{l}}-f_k^\alpha A_{\bar{l}}^k \theta^{\bar{l}} \wedge \theta-f_{\bar{k}}^\alpha A_l^{\bar{k}} \theta^l \wedge \theta=0,
\end{equation}
where
 \begin{align}
    D f_k^\alpha & \equiv d f_k^\alpha-f_l^\alpha \theta_k^l+f_k^\beta \widetilde{\omega}_\beta^\alpha=f_{k B}^\alpha \theta^B,  \label{eq:df1} \\
    D f_{\bar{k}}^\alpha & \equiv d f_{\bar{k}}^\alpha-f_{\bar{l}}^\alpha \theta_{\bar{k}}^{\bar{l}}+f_{\bar{k}}^\beta \widetilde{\omega}_\beta^\alpha=f_{\bar{k} B}^\alpha \theta^B, \label{eq:df2}\\
    D f_0^\alpha & \equiv d f_0^\alpha+f_0^\beta \widetilde{\omega}_\beta^\alpha=f_{0 B}^\alpha \theta^B . \label{eq:df3} 
    \end{align}
    Here, for simplicity, we write  $f^* ( \widetilde{\omega}_\beta^\alpha )$ as  $ \widetilde{\omega}_\beta^\alpha$   on the right hand side  of the above formulas.     
Then $ \eqref{eq:Df} $ gives 
\begin{equation}\label{eq:cmu1}
    f_{jk}^{\alpha }=f_{kj}^{\alpha }, \quad f_{\bar j \bar k}^{\alpha }=f_{\bar k \bar j}^{\alpha }, \quad f_{j \bar k}^{\alpha }-f_{\bar k j}^{\alpha }=2 \sqrt{-1}f_{0}^{\alpha }h_{j\bar{k}}, \quad f_{0 j }^{\alpha }-f_{j 0}^{\alpha }=f^{\alpha }_{\bar{k}}A^{\bar{k}}_{j},\quad f_{0 \bar j }^{\alpha }-f_{\bar j 0}^{\alpha }=f^{\alpha }_{k}A^{k}_{\bar j} .
\end{equation}
Note that here and in the following, we have $h_{j\bar k}=\delta _{j  k}$, since we have adopted a unitary frame.

Differentiating $ \eqref{eq:df1} $, we have 
$$ -f_l^\alpha d \theta_k^l+f_k^\beta d \widetilde{\omega}_\beta^\alpha -d f_l^\alpha \wedge  \theta_k^l+df_k^\beta \wedge \widetilde{\omega}_\beta^\alpha =f_{k B}^\alpha d \theta^B +df_{k B}^\alpha \wedge  \theta^B. $$
Using structure equations again, we have 
$$ \begin{aligned}
   0=& f_{j}^{\alpha }(-\theta _{l}^{j}\wedge \theta _{k}^{l}+\Pi_{k}^{j})-f_{k}^{\beta }(-\widetilde{\omega}_\gamma^\alpha \wedge \widetilde{\omega}_\beta^\gamma+\widetilde{\Omega}_\beta^\alpha)\\
    &+f_{k j}^{\alpha }(\theta^l\wedge\theta_l^j+\theta\wedge\tau^j )+f_{k \bar j}^{\alpha }(\theta^{\bar l}\wedge\theta_{\bar l}^{\bar j}+\theta\wedge\tau^{\bar j} )+2\sqrt{-1}h_{j\bar{k}}f_{k 0}^{\alpha } \theta^{j}\wedge\theta^{\bar k}\\
    &+ d f_l^\alpha \wedge  \theta_k^l-df_k^\beta \wedge \widetilde{\omega}_\beta^\alpha +df_{k B}^\alpha \wedge  \theta^B.
\end{aligned}  $$
It follows that 
\begin{equation}\label{eq:DDf}
D f_{k B}^{\alpha}  \wedge \theta^B+2 \sqrt{-1} f_{k 0}^{\alpha} h_{j \bar{l}} \theta^j \wedge \theta^{\bar{l}}-f_{k l}^{\alpha} A_{\bar{j}}^l \theta^{\bar{j}} \wedge \theta-f_{k \bar{l}}^{\alpha} A_j^{\bar{l}} \theta^j \wedge \theta=-f_l^{\alpha} \Pi_k^l+f_k^{\beta } \widetilde{\Omega}_{\beta }^{\alpha},
\end{equation}
where
\begin{align}
    Df_{jk}^{\alpha } &\equiv df _{jk}^{\alpha }-f_{jl}^{\alpha}\theta ^{l}_{k}-f_{lk}^{\alpha}\theta ^{l}_{j}+f_{jk}^{\beta }\widetilde{\omega} _{\beta }^{\alpha }=f_{jkB}^{\alpha }\theta ^{B},\\
    Df_{j \bar k}^{\alpha } &\equiv df _{j\bar k}^{\alpha }-f_{j\bar l}^{\alpha}\theta ^{\bar l}_{\bar k}-f_{l \bar k}^{\alpha}\theta ^{l}_{j}+f_{j \bar k}^{\beta }\widetilde{\omega} _{\beta }^{\alpha }=f_{j \bar kB}^{\alpha }\theta ^{B},\\
    Df_{j0}^{\alpha } &\equiv df _{j0}^{\alpha }-f_{l0}^{\alpha}\theta ^{l}_{j}+f_{j0}^{\beta }\widetilde{\omega } _{\beta }^{\alpha }=f_{j0B}^{\alpha }\theta ^{B}. \label{eq:DDf3}
\end{align} 
From $ \eqref{eq:DDf} $, we have
\begin{equation}\label{cmu32}
    \begin{aligned}
        & f_{ijk}^{\alpha }=f_{ikj}^{\alpha }-f_{i}^{\beta }f_{j}^{\gamma }f_{k}^{\bar \sigma }\widetilde{R}_{\beta \gamma \bar \sigma }^{\alpha }+f_{i}^{\beta }f_{k}^{\gamma }f_{j}^{\bar \sigma }\widetilde{R}_{\beta \gamma \bar \sigma }^{\alpha } + 2 \sqrt{-1}f_{j}^{\alpha }A_{ik}-2 \sqrt{-1}f_{k}^{\alpha }A_{ij}, \\
        & f_{i \bar j \bar k}^{\alpha }=f_{i\bar k\bar j}^{\alpha }-f_{i}^{\beta }f_{\bar j}^{\gamma }f_{\bar k}^{\bar \sigma }\widetilde{R}_{\beta \gamma \bar \sigma }^{\alpha }+f_{i}^{\beta }f_{\bar k}^{\gamma }f_{\bar j}^{\bar \sigma }\widetilde{R}_{\beta \gamma \bar \sigma }^{\alpha } + 2 \sqrt{-1}f_{l}^{\alpha }h_{i \bar{j}}A_{\bar{k}}^{l}- 2 \sqrt{-1}f_{l}^{\alpha }h_{i \bar{k}}A_{\bar{j}}^{l},\\
        & f_{ij\bar k}^{\alpha }=f_{i\bar kj}^{\alpha }-f_{i}^{\beta }f_{j}^{\gamma }f_{\bar k}^{\bar \sigma }\widetilde{R}_{\beta \gamma \bar \sigma }^{\alpha }+f_{i}^{\beta }f_{\bar k}^{\gamma }f_{j}^{\bar \sigma }\widetilde{R}_{\beta \gamma \bar \sigma }^{\alpha }+f_{l}^{\alpha }R_{ij \bar{k}}^{l}+2 \sqrt{-1}f_{i0}^{\alpha }h_{j \bar{k}}, \\
        & f_{ij0}^{\alpha }=f_{i0j}^{\alpha }-f_{i}^{\beta }f_{j}^{\gamma }f_{0}^{\bar \sigma }\widetilde{R}_{\beta \gamma \bar \sigma }^{\alpha }+f_{i}^{\beta }f_{0}^{\gamma }f_{j}^{\bar \sigma }\widetilde{R}_{\beta \gamma \bar \sigma }^{\alpha } +f_{l}^{\alpha }h^{l \bar{k}}A_{ij , \bar{k}}-f_{i \bar{k}}^{\alpha }A_{j}^{\bar{k}},\\
        & f_{i\bar j0}^{\alpha }=f_{i0\bar j}^{\alpha }-f_{i}^{\beta }f_{\bar j}^{\gamma }f_{0}^{\bar \sigma }\widetilde{R}_{\beta \gamma \bar \sigma }^{\alpha }+f_{i}^{\beta }f_{0}^{\gamma }f_{\bar j}^{\bar \sigma }\widetilde{R}_{\beta \gamma \bar \sigma }^{\alpha } -f_{l}^{\alpha }h^{l \bar{k}}A_{\bar j  \bar{k},i}-f_{i k}^{\alpha }A_{\bar j}^{k}.
    \end{aligned} 
\end{equation}

Similarly, differentiating $ \eqref{eq:df2} $, we have 
\begin{equation}\label{eq:DDfb}
    D f_{\bar k B}^{\alpha}  \wedge \theta^B+2 \sqrt{-1} f_{\bar k 0}^{\alpha} h_{j \bar{l}} \theta^j \wedge \theta^{\bar{l}}-f_{\bar k l}^{\alpha} A_{\bar{j}}^l \theta^{\bar{j}} \wedge \theta-f_{\bar k \bar{l}}^{\alpha} A_j^{\bar{l}} \theta^j \wedge \theta=-f_{\bar l }^{\alpha} \Pi_{\bar k }^{\bar l}+f_{\bar k }^{\beta } \widetilde{\Omega}_{\beta }^{\alpha},
    \end{equation}
where 
\begin{align}
    Df_{\bar j k}^{\alpha } &\equiv df _{\bar jk}^{\alpha }-f_{\bar jl}^{\alpha}\theta ^{l}_{k}-f_{\bar l k}^{\alpha}\theta ^{\bar l}_{\bar j}+f_{\bar jk}^{\beta }\widetilde{\omega} _{\beta }^{\alpha }=f_{\bar jkB}^{\alpha }\theta ^{B},\\
    Df_{\bar j \bar k}^{\alpha } &\equiv df _{\bar j\bar k}^{\alpha }-f_{\bar j\bar l}^{\alpha}\theta ^{\bar l}_{\bar k}-f_{\bar l \bar k}^{\alpha}\theta ^{\bar l}_{\bar j}+f_{\bar j \bar k}^{\beta }\widetilde{\omega} _{\beta }^{\alpha }=f_{\bar j \bar kB}^{\alpha }\theta ^{B},\\
    Df_{\bar j0}^{\alpha } &\equiv df _{\bar j0}^{\alpha }-f_{\bar l0}^{\alpha}\theta ^{\bar l}_{\bar j}+f_{\bar j0}^{\beta }\widetilde{\omega }  _{\beta }^{\alpha }=f_{\bar j0B}^{\alpha }\theta ^{B}. 
\end{align}
From $ \eqref{eq:DDfb} $, we have 
\begin{equation}\label{eq:c37}
    \begin{aligned}
        & f_{\bar ijk}^{\alpha }=f_{\bar ikj}^{\alpha }-f_{\bar i}^{\beta }f_{j}^{\gamma }f_{k}^{\bar \sigma }\widetilde{R}_{\beta \gamma \bar \sigma }^{\alpha }+f_{\bar i}^{\beta }f_{k}^{\gamma }f_{j}^{\bar \sigma }\widetilde{R}_{\beta \gamma \bar \sigma }^{\alpha } + 2 \sqrt{-1}f_{\bar l}^{\alpha }h_{\bar i k}A_{j}^{\bar l}- 2 \sqrt{-1}f_{\bar l}^{\alpha }h_{\bar i j}A_{k}^{\bar l}, \\
        & f_{\bar i \bar j \bar k}^{\alpha }=f_{\bar i\bar k\bar j}^{\alpha }-f_{\bar i}^{\beta }f_{\bar j}^{\gamma }f_{\bar k}^{\bar \sigma }\widetilde{R}_{\beta \gamma \bar \sigma }^{\alpha }+f_{\bar i}^{\beta }f_{\bar k}^{\gamma }f_{\bar j}^{\bar \sigma }\widetilde{R}_{\beta \gamma \bar \sigma }^{\alpha } + 2 \sqrt{-1}f_{\bar k}^{\alpha }A_{\bar i \bar j}- 2 \sqrt{-1}f_{\bar j}^{\alpha }A_{\bar i \bar k},\\
        & f_{\bar ij\bar k}^{\alpha }=f_{\bar i\bar kj}^{\alpha }-f_{\bar i}^{\beta }f_{j}^{\gamma }f_{\bar k}^{\bar \sigma }\widetilde{R}_{\beta \gamma \bar \sigma }^{\alpha }+f_{\bar i}^{\beta }f_{\bar k}^{\gamma }f_{j}^{\bar \sigma }\widetilde{R}_{\beta \gamma \bar \sigma }^{\alpha }+f_{\bar l}^{\alpha }R_{\bar ij \bar{k}}^{\bar l}+2 \sqrt{-1}f_{\bar i0}^{\alpha }h_{j \bar{k}} ,\\
        & f_{\bar ij0}^{\alpha }=f_{\bar i0j}^{\alpha }-f_{\bar i}^{\beta }f_{j}^{\gamma }f_{0}^{\bar \sigma }\widetilde{R}_{\beta \gamma \bar \sigma }^{\alpha }+f_{\bar i}^{\beta }f_{0}^{\gamma }f_{j}^{\bar \sigma }\widetilde{R}_{\beta \gamma \bar \sigma }^{\alpha } -f_{\bar l}^{\alpha }h^{\bar l k}A_{jk,\bar i}-f_{\bar i \bar{k}}^{\alpha }A_{j}^{\bar{k}},\\
        & f_{\bar i\bar j0}^{\alpha }=f_{\bar i0\bar j}^{\alpha }-f_{\bar i}^{\beta }f_{\bar j}^{\gamma }f_{0}^{\bar \sigma }\widetilde{R}_{\beta \gamma \bar \sigma }^{\alpha }+f_{\bar i}^{\beta }f_{0}^{\gamma }f_{\bar j}^{\bar \sigma }\widetilde{R}_{\beta \gamma \bar \sigma }^{\alpha } +f_{\bar l}^{\alpha }h^{\bar l k}A_{\bar i \bar j ,k}-f_{\bar i k}^{\alpha }A_{\bar j}^{k}.
    \end{aligned} 
\end{equation}

Using the same argument again, differentiating  $ \eqref{eq:df3} $ yields
\begin{equation}\label{eq:DDf2}
    D f_{0B}^{\alpha } \wedge \theta ^{B}+2 \sqrt{-1} f_{00}^{\alpha }h_{j \bar k} \theta ^{j} \wedge \theta ^{\bar k}-f_{0 j}^{\alpha } A_{\bar k}^{j} \theta ^{\bar k} \wedge \theta -f_{0 \bar j}^{\alpha } A_{ k}^{\bar j} \theta ^{ k} \wedge \theta=f_{0}^{\beta } \widetilde{\Omega }_{\beta }^{\alpha } ,
\end{equation}
where 
\begin{align}
    Df_{0 k}^{\alpha } \equiv& d f_{0 k}^{\alpha }-f_{0 j}^{\alpha } \theta ^{j}_{k}+f_{0 k}^{\beta } \widetilde{\omega }_{\beta }^{\alpha }=f_{0 k B}^{\alpha }\theta ^B,\\
    Df_{0 \bar k}^{\alpha } \equiv& d f_{0 \bar k}^{\alpha }-f_{0 \bar j}^{\alpha } \theta ^{\bar j}_{\bar k}+f_{0 \bar k}^{\beta } \widetilde{\omega }_{\beta }^{\alpha }=f_{0 \bar k B}^{\alpha }\theta ^B,\\
    Df_{0 0}^{\alpha } \equiv& d f_{0 0}^{\alpha }+f_{0 0}^{\beta } \widetilde{\omega }_{\beta }^{\alpha }=f_{0 0 B}^{\alpha }\theta ^B.
\end{align} 
From $ \eqref{eq:DDf2} $, we have 
\begin{equation}\label{cmu43}
    \begin{aligned}
        f_{0 j k}^{\alpha }=& f_{0 k j}^{\alpha }-f_{0}^{\beta }f_{j}^{\gamma }f_{k}^{\bar \sigma } \widetilde{R}_{\beta \gamma \bar \sigma }^{\alpha }+f_{0}^{\beta }f_{k}^{\gamma }f_{j}^{\bar \sigma } \widetilde{R}_{\beta \gamma \bar \sigma }^{\alpha }  ,\\
        f_{0 j \bar k}^{\alpha }=& f_{0 \bar k j}^{\alpha }-f_{0}^{\beta }f_{j}^{\gamma }f_{\bar k}^{\bar \sigma } \widetilde{R}_{\beta \gamma \bar \sigma }^{\alpha }+f_{0}^{\beta }f_{\bar k}^{\gamma }f_{j}^{\bar \sigma } \widetilde{R}_{\beta \gamma \bar \sigma }^{\alpha } +2 \sqrt{-1}f_{00}^{\alpha }h_{j \bar k}, \\
        f_{0 0 k}^{\alpha }=& f_{0 k 0}^{\alpha }-f_{0}^{\beta }f_{0}^{\gamma }f_{k}^{\bar \sigma } \widetilde{R}_{\beta \gamma \bar \sigma }^{\alpha }+f_{0}^{\beta }f_{k}^{\gamma }f_{0}^{\bar \sigma } \widetilde{R}_{\beta \gamma \bar \sigma }^{\alpha }+f_{0 \bar j}^{\alpha }A_{k}^{\bar j}  ,\\
        f_{0 0 \bar k}^{\alpha }=& f_{0 \bar k 0}^{\alpha }-f_{0}^{\beta }f_{0}^{\gamma }f_{\bar k}^{\bar \sigma } \widetilde{R}_{\beta \gamma \bar \sigma }^{\alpha }+f_{0}^{\beta }f_{\bar k}^{\gamma }f_{0}^{\bar \sigma } \widetilde{R}_{\beta \gamma \bar \sigma }^{\alpha }+f_{0 j}^{\alpha }A_{\bar k}^{ j} . 
    \end{aligned} 
\end{equation}

Last, from $ \eqref{eq:cmu1} $, we have
\begin{equation}\label{eq:c43}
    \begin{aligned}
        f_{i \bar j k}^{\alpha }=&f_{\bar j i k}^{\alpha }+2 \sqrt{-1}h_{i \bar j}f_{0k}^{\alpha }, \\
        f_{i \bar j \bar k}^{\alpha }=&f_{\bar j i \bar k}^{\alpha }+2 \sqrt{-1}h_{i \bar j}f_{0 \bar k}^{\alpha } ,\\
        f_{0  j k}^{\alpha }=&f_{j 0 k}^{\alpha }+f_{\bar l k}^{\alpha }A_{j}^{\bar l}+f_{\bar l}^{\alpha }A_{j,k}^{\bar l},\\
        f_{0  j \bar k}^{\alpha }=&f_{j 0 \bar k}^{\alpha }+f_{\bar l \bar k}^{\alpha }A_{j}^{\bar l}+f_{\bar l}^{\alpha }A_{j,\bar k}^{\bar l},\\
        f_{0 \bar j k}^{\alpha }=&f_{\bar j 0 k}^{\alpha }+f^{\alpha }_{l k}A^{l}_{\bar j}+f^{\alpha }_{ l}A^{l}_{\bar j,k},\\
        f_{0 \bar j \bar k}^{\alpha }=&f_{\bar j 0 \bar k}^{\alpha }+f^{\alpha }_{l \bar k}A^{l}_{\bar j}+f^{\alpha }_{ l}A^{l}_{\bar j,\bar k}.
    \end{aligned} 
\end{equation}

\section{Foliated and $(J,\widetilde{J})$-holomorphicity results}
A divergence of a vector field $X $ on $(M,H , \theta ) $ is defined by 
$$ L_{X}\Psi =\di(X)\Psi  ,$$
where $\Psi =\theta \wedge (d \theta )^m $ is the volume form.
One has (cf. Lemma \ref{lemDiv})
\begin{equation}
    \di(X)=\trace_{g_\theta } (Y \in TM \to \nabla _YX).
\end{equation}
Also note that $\di $ is a real operator:
\begin{equation}
    \overline{\di(X)}=\di(\bar X).
\end{equation}

    If $u $ is a function on $(M,H ,\theta ) $, then its sub-Laplacian $\lap_b $ is defined by, under an  adapted frame,
    $$ \lap_b u:=\di(\nabla ^H u)= u_{i \bar{i }}+u_{ \bar{i }i},$$
    where $\nabla ^H u $ is the horizontal component of the gradient  of $u $.
    Note that the usual Laplacian of $u $ is 
$$ \lap u= u_{i \bar{i }}+u_{ \bar{i }i}+u_{00}.$$

Using an adapted frame, we can express $\tau _{\dbb,\xi }(f)$ as follows: 
$$ \tau _{\dbb,\xi}(f)=(f_{j \bar j }^{\alpha  }+f_{\bar j j }^{\alpha  }+f_{00 }^{\alpha  }-2m \sqrt{-1} f_{0 }^{\alpha  })\tilde{\eta} _{\alpha }+(f_{j \bar j }^{\bar \alpha  }+f_{\bar j j }^{\bar \alpha  }+f_{00 }^{\bar \alpha  }+2m \sqrt{-1} f_{0 }^{\bar \alpha  })\tilde{\eta} _{\bar \alpha }.$$
Besides, it follows from  the third equation of $\eqref{eq:cmu1}$ that
\begin{equation*}
    f_{j \bar j }^{\alpha  }+f_{\bar j j }^{\alpha  }+f_{00 }^{\alpha  }-2m \sqrt{-1} f_{0 }^{\alpha  }=2f_{\bar jj}^{\alpha}+f_{00}^{\alpha}.
\end{equation*}
Therefore, defining $(Lf)^{\alpha }:=2f_{\bar jj}^{\alpha}+f_{00}^{\alpha} $, we may express $\tau _{\dbb,\xi }(f)$ as
\begin{equation}\label{Lf}
    \tau _{\dbb,\xi }(f)=(Lf)^{\alpha } \widetilde{\eta }_{\alpha } +\overline{ (Lf)^{\alpha }} \widetilde{\eta }_{\bar \alpha }.
\end{equation}

By applying the commutation relations in \S \ref{se:commut}, we have 
\begin{lemma}\label{lem:bochner0}
    \begin{equation}\label{la1}
        \begin{aligned}
            \frac{1}{2} \lap |df(\xi )|^2
            =& 2(|f_{0j }^{\alpha  }|^2+|f_{0 \bar{j} }^{\alpha  }|^2+|f_{00 }^{\alpha  }|^2)+ f_{0 }^{\bar{\alpha } }(Lf)^{\alpha }_{0} + f_{0 }^{\alpha  }\overline{(Lf)^{\alpha }_{0}} +2 \sqrt{-1} m  (f_{0 }^{\bar \alpha  }f_{00 }^{\alpha  }-f_{0 }^{ \alpha  }f_{00 }^{\bar\alpha  })\\
            &+2 f_{0 }^{\bar{\alpha } } f_{\bar j}^{\beta }f_{j}^{\gamma }f_{0}^{\bar \sigma }\widetilde{R}_{\bar{\alpha }\beta \gamma \bar \sigma }+2 f_{\bar{j} }^{\bar{\alpha } } f_{0}^{\beta }f_{0}^{\gamma }f_{j}^{\bar \sigma }\widetilde{R}_{\bar{\alpha }\beta \gamma \bar \sigma } -2f_{0 }^{\bar{\alpha } } f_{\bar j}^{\beta }f_{0}^{\gamma }f_{j}^{\bar \sigma }\widetilde{R}_{\bar{\alpha }\beta \gamma \bar \sigma }-2f_{0 }^{\bar{\alpha } }f_{j}^{\beta }f_{0}^{\gamma }f_{\bar j}^{\bar \sigma }\widetilde{R}_{\bar{\alpha }\beta \gamma \bar \sigma }\\
            &+2 (f_{0 }^{\bar{\alpha } } f_{\bar{l} }^{\alpha  }+ f_{0 }^{\alpha  } f_{\bar{l} }^{\bar \alpha  })A_{j,\bar{j} }^{\bar{l} }+2 (f_{0 }^{\bar{\alpha } }f^{\alpha }_{ l}+f_{0 }^{\alpha  }f^{\bar \alpha }_{ l})A^{l}_{\bar j,j}\\
            &+2( f_{0 }^{\bar{\alpha } }f_{\bar j \bar{k}}^{\alpha }+ f_{0 }^{\alpha  }f_{\bar j \bar{k}}^{\bar \alpha })A_{j}^{\bar{k}}+ 2 (f_{0 }^{\bar{\alpha } } f^{\alpha }_{l j}+f_{0 }^{\alpha  } f^{\bar \alpha }_{l j})A^{l}_{\bar j}.
        \end{aligned} 
    \end{equation}
\end{lemma}
\begin{proof}
    First,
\begin{equation}\label{f032}
    \begin{aligned}
     \frac{1}{2} \lap |df(\xi )|^2
        =& (f_0^{\alpha }f_0^{\bar{\alpha }})_{j \bar{j}}+(f_0^{\alpha }f_0^{\bar{\alpha }})_{\bar{j}j} +(f_0^{\alpha }f_0^{\bar{\alpha }})_{00}\\
        =& 2(f_{0j }^{\alpha  }f_{0 \bar{j} }^{\bar{\alpha } }+f_{0 \bar j }^{ \alpha  }f_{0 j }^{\bar{\alpha } }+f_{00}^{\alpha }f_{00}^{\bar \alpha })+f_{0 }^{\bar \alpha  }(f_{0  \bar{j} j }^{\alpha  }+ f_{0 j\bar{j} }^{\alpha  }+f_{000}^{\alpha })+f_{0 }^{\alpha  }(f_{0 j \bar{j} }^{\bar \alpha  }+ f_{0 \bar{j} j}^{\bar{\alpha } }+f_{000}^{\bar \alpha }).
    \end{aligned} 
\end{equation}
From $\eqref{eq:c43} $ and $\eqref{eq:c37} $, we have
\begin{equation}\label{f033}
    \begin{aligned}
        f_{0 \bar j j}^{\alpha }=&f_{\bar j 0 j}^{\alpha }+f^{\alpha }_{l j}A^{l}_{\bar j}+f^{\alpha }_{ l}A^{l}_{\bar j,j}\\
        =& f_{\bar j  j 0}^{\alpha } +f_{\bar j}^{\beta }f_{j}^{\gamma }f_{0}^{\bar \sigma }\widetilde{R}_{\beta \gamma \bar \sigma }^{\alpha }-f_{\bar j}^{\beta }f_{0}^{\gamma }f_{j}^{\bar \sigma }\widetilde{R}_{\beta \gamma \bar \sigma }^{\alpha } \\
        &+f_{\bar l}^{\alpha }h^{\bar l k}A_{jk,\bar j}+f_{\bar j \bar{k}}^{\alpha }A_{j}^{\bar{k}}+f^{\alpha }_{l j}A^{l}_{\bar j}+f^{\alpha }_{ l}A^{l}_{\bar j,j}.
    \end{aligned} 
\end{equation}
From $\eqref{eq:c43} $ and $\eqref{cmu32} $, we have
\begin{equation}\label{f034}
    \begin{aligned}
        f_{0  j \bar j}^{\alpha }=&f_{j 0 \bar j}^{\alpha }+f_{\bar l \bar j}^{\alpha }A_{j}^{\bar l}+f_{\bar l}^{\alpha }A_{j,\bar j}^{\bar l}\\
        =&  f_{j\bar j0}^{\alpha }+f_{j}^{\beta }f_{\bar j}^{\gamma }f_{0}^{\bar \sigma }\widetilde{R}_{\beta \gamma \bar \sigma }^{\alpha }-f_{j}^{\beta }f_{0}^{\gamma }f_{\bar j}^{\bar \sigma }\widetilde{R}_{\beta \gamma \bar \sigma }^{\alpha } \\
        &+f_{l}^{\alpha }h^{l \bar{k}}A_{\bar j  \bar{k},j}+f_{j k}^{\alpha }A_{\bar j}^{k} +f_{\bar l \bar j}^{\alpha }A_{j}^{\bar l}+f_{\bar l}^{\alpha }A_{j,\bar j}^{\bar l} .
    \end{aligned} 
\end{equation}
Note that
\begin{equation}\label{f035}
    \begin{aligned}
        &f_{0 }^{\bar{\alpha } } (f_{\bar j}^{\beta }f_{j}^{\gamma }f_{0}^{\bar \sigma }\widetilde{R}_{\beta \gamma \bar \sigma }^{\alpha }-f_{\bar j}^{\beta }f_{0}^{\gamma }f_{j}^{\bar \sigma }\widetilde{R}_{\beta \gamma \bar \sigma }^{\alpha }+f_{j}^{\beta }f_{\bar j}^{\gamma }f_{0}^{\bar \sigma }\widetilde{R}_{\beta \gamma \bar \sigma }^{\alpha }-f_{j}^{\beta }f_{0}^{\gamma }f_{\bar j}^{\bar \sigma }\widetilde{R}_{\beta \gamma \bar \sigma }^{\alpha }) \\
        =&2 f_{0 }^{\bar{\alpha } } f_{\bar j}^{\beta }f_{j}^{\gamma }f_{0}^{\bar \sigma }\widetilde{R}_{\bar{\alpha }\beta \gamma \bar \sigma } -f_{0 }^{\bar{\alpha } } f_{\bar j}^{\beta }f_{0}^{\gamma }f_{j}^{\bar \sigma }\widetilde{R}_{\bar{\alpha }\beta \gamma \bar \sigma }-f_{0 }^{\bar{\alpha } }f_{j}^{\beta }f_{0}^{\gamma }f_{\bar j}^{\bar \sigma }\widetilde{R}_{\bar{\alpha }\beta \gamma \bar \sigma },
    \end{aligned} 
\end{equation}
and, by $\eqref{eq:cmu1} $,
\begin{equation}\label{f036}
    \begin{aligned}
      &f_{0 }^{\bar{\alpha } }(f_{\bar l}^{\alpha }h^{\bar l k}A_{jk,\bar j}+f_{\bar j \bar{k}}^{\alpha }A_{j}^{\bar{k}}+f^{\alpha }_{l j}A^{l}_{\bar j}+f^{\alpha }_{ l}A^{l}_{\bar j,j})  \\
      &+f_{0 }^{\bar{\alpha } }(f_{l}^{\alpha }h^{l \bar{k}}A_{\bar j  \bar{k},j}+f_{j k}^{\alpha }A_{\bar j}^{k} +f_{\bar l \bar j}^{\alpha }A_{j}^{\bar l}+f_{\bar l}^{\alpha }A_{j,\bar j}^{\bar l}) \\
      =& 2 f_{0 }^{\bar{\alpha } } f_{\bar{l} }^{\alpha  }A_{j,\bar{j} }^{\bar{l} }+2 f_{0 }^{\bar{\alpha } }f^{\alpha }_{ l}A^{l}_{\bar j,j}+2 f_{0 }^{\bar{\alpha } }f_{\bar j \bar{k}}^{\alpha }A_{j}^{\bar{k}}+ 2 f_{0 }^{\bar{\alpha } } f^{\alpha }_{l j}A^{l}_{\bar j}.
    \end{aligned} 
\end{equation}
Therefore, substituting $\eqref{f033},\eqref{f034},\eqref{f035},\eqref{f036}$ into $\eqref{f032}$, we get
\begin{equation*}
    \begin{aligned}
        \frac{1}{2} \lap |df(\xi )|^2=& 2(|f_{0j }^{\alpha  }|^2+|f_{0 \bar{j} }^{\alpha  }|^2+|f_{00}^{\alpha }|^2)+ f_{0 }^{\bar{\alpha } }(f_{\bar j  j 0}^{\alpha }+f_{j \bar j   0}^{\alpha }+f_{000}^{\alpha }) + f_{0 }^{\alpha  }(f_{\bar j  j 0}^{\bar \alpha }+f_{j \bar j   0}^{\bar \alpha }+f_{000}^{\bar \alpha })\\
        &+2 f_{0 }^{\bar{\alpha } } f_{\bar j}^{\beta }f_{j}^{\gamma }f_{0}^{\bar \sigma }\widetilde{R}_{\bar{\alpha }\beta \gamma \bar \sigma }+2 f_{\bar{j} }^{\bar{\alpha } } f_{0}^{\beta }f_{0}^{\gamma }f_{j}^{\bar \sigma }\widetilde{R}_{\bar{\alpha }\beta \gamma \bar \sigma } -2f_{0 }^{\bar{\alpha } } f_{\bar j}^{\beta }f_{0}^{\gamma }f_{j}^{\bar \sigma }\widetilde{R}_{\bar{\alpha }\beta \gamma \bar \sigma }-2f_{0 }^{\bar{\alpha } }f_{j}^{\beta }f_{0}^{\gamma }f_{\bar j}^{\bar \sigma }\widetilde{R}_{\bar{\alpha }\beta \gamma \bar \sigma }\\
        &+2 (f_{0 }^{\bar{\alpha } } f_{\bar{l} }^{\alpha  }+ f_{0 }^{\alpha  } f_{\bar{l} }^{\bar \alpha  })A_{j,\bar{j} }^{\bar{l} }+2 (f_{0 }^{\bar{\alpha } }f^{\alpha }_{ l}+f_{0 }^{\alpha  }f^{\bar \alpha }_{ l})A^{l}_{\bar j,j}\\
        &+2( f_{0 }^{\bar{\alpha } }f_{\bar j \bar{k}}^{\alpha }+ f_{0 }^{\alpha  }f_{\bar j \bar{k}}^{\bar \alpha })A_{j}^{\bar{k}}+ 2 (f_{0 }^{\bar{\alpha } } f^{\alpha }_{l j}+f_{0 }^{\alpha  } f^{\bar \alpha }_{l j})A^{l}_{\bar j}.
    \end{aligned} 
\end{equation*}
Taking into account the identity
$$ (Lf)^{\alpha }_0=f_{\bar j  j 0}^{\alpha }+f_{j \bar j   0}^{\alpha } +f_{000}^{\alpha }-2m \sqrt{-1}f_{00}^{\alpha },$$
we obtain $\eqref{la1}$.

\end{proof}

\begin{remark}\label{re:2.2}
    One can check that 
    $$ \begin{aligned}
        \widetilde{g}\left(\widetilde{R}(d f(\eta _j),df(\xi ))\overline{d f(\eta _j)},df(\xi )\right)=& \widetilde{g}\left(\widetilde{R}\left(f_{j }^{\beta  }\tilde{\eta }_{\beta }+f_{j }^{\bar \alpha  }\tilde{\eta }_{\bar \alpha  },f_{0 }^{\gamma  }\tilde{\eta }_{\gamma }+f_{0 }^{\bar \sigma   }\tilde{\eta }_{\bar \sigma  }\right) \left(f_{\bar j }^{\bar \alpha   }\tilde{\eta }_{\bar \alpha  }+f_{\bar j }^{\beta   }\tilde{\eta }_{ \beta  }\right) ,f_{0 }^{\gamma  }\tilde{\eta }_{\gamma }+f_{0 }^{\bar \sigma   }\tilde{\eta }_{\bar \sigma  }\right) \\
        =&  f_{0 }^{\bar{\alpha } } f_{\bar j}^{\beta }f_{j}^{\gamma }f_{0}^{\bar \sigma }\widetilde{R}_{\bar{\alpha }\beta \gamma \bar \sigma }+ f_{\bar{j} }^{\bar{\alpha } } f_{0}^{\beta }f_{0}^{\gamma }f_{j}^{\bar \sigma }\widetilde{R}_{\bar{\alpha }\beta \gamma \bar \sigma } \\
        &-f_{0 }^{\bar{\alpha } } f_{\bar j}^{\beta }f_{0}^{\gamma }f_{j}^{\bar \sigma }\widetilde{R}_{\bar{\alpha }\beta \gamma \bar \sigma }-f_{0 }^{\bar{\alpha } }f_{j}^{\beta }f_{0}^{\gamma }f_{\bar j}^{\bar \sigma }\widetilde{R}_{\bar{\alpha }\beta \gamma \bar \sigma }.
    \end{aligned}  $$
If $N $ has non-positive sectional curvature, then
$$ \widetilde{g}(\widetilde{R}(Z,X)\overline{Z},X) \geq 0 $$
for any complex vector $Z $ and  any real vector $X $ on $N $. Thus, if this is the case, the second line in the right hand side of $\eqref{la1} $ is non-negative.

\end{remark}

\begin{lemma}\label{le2.5}
    Let $(M^{2m+1},H,J, \theta ) $  be a compact pseudo-Hermitian manifold.
    Let $f :M^{2m+1} \to (N^{2n},\widetilde{J},\widetilde{g})$ be a smooth map. If the second fundamental form satisfies
    $$ \beta (\xi ,X)=0, \quad \text{for any } X \in H  ,$$
    then $f $ is foliated.
\end{lemma}
\begin{proof}
    Since $N $ is a Riemannian manifold, the claim follows directly from  \cite{CDRY2019onharm}. We present the proof for readers' convenience. 

    By the integration by parts and the third  formula in $\eqref{eq:cmu1} $, we have 
    $$ \begin{aligned}
        0=\sqrt{-1} \int_{M}( f_{j }^{\alpha  }f_{0\bar j}^{\bar \alpha  }-f_{\bar j }^{ \alpha  }f_{0 j}^{ \bar \alpha  } )=&-\sqrt{-1}\int_{M} (f_{j \bar j}^{\alpha  }f_{0}^{\bar \alpha  }-f_{\bar j j}^{ \alpha  }f_{0 }^{ \bar \alpha  }) \\
        =& 2m\int_{M}|f_{0 }^{\alpha  }|^2.
    \end{aligned}  $$
    Therefore, $f_{0 }^{\alpha  } =0$.
\end{proof}

The main difficulty in applying Lemma \ref{lem:bochner0} arises from the mixed term
$$
2 \sqrt{-1} m \left( f_{0}^{\bar{\alpha}} f_{00}^{\alpha} - f_{0}^{\alpha} f_{00}^{\bar{\alpha}} \right)
$$
and the terms related to torsion.
To address the mixed term, we need to add an extra term \( |f_{00}^{\alpha}|^2 \) (see below for details). 
Inspired by \cite{CDRY2019onharm}, we define the following generalized Paneitz operator acting on maps:
    $$ Pf:= \underbrace{(f_{\bar j j k}^{\alpha }+\frac{1}{2}f_{00k}^{\alpha }+
    2m \sqrt{-1}A_{k j}f_{\bar j}^{\alpha }) }_{:=(Pf)_{k}^{\alpha }}\theta ^k \otimes \widetilde{\eta }_{\alpha } .$$

In \cite{LS2019CRanalogue} (see also  \cite{grahamSmoothSolutionsDegenerate1988}),  Li and Son
defined  the following tensors $$Bf=B_{i \bar j}f^{\alpha } \theta ^{i} \otimes  \theta ^{\bar j} \otimes \widetilde{\eta }_{\alpha }$$ 
and
$$E=E_{\bar j} \theta ^{\bar j},$$
where  
$$ B_{i \bar j}f^{\alpha }:=f_{i \bar j}^{\alpha }-\frac{1}{m}f_{k \bar k}^{\alpha }h_{i \bar j} $$
and 
$$ E_{\bar j}:= (B_{i \bar j}f^{\alpha }) f_{\bar i}^{\bar \alpha }.$$
Then  $-\delta E$ is given by
$$ \begin{aligned}
    E_{\bar j ,j}=& (f_{i \bar j j}^{\alpha }-\frac{1}{m}f_{k \bar k j}^{\alpha }h_{i \bar j}) f_{\bar i}^{\bar \alpha }+(B_{i \bar j}f^{\alpha }) f_{\bar i j}^{\bar \alpha }\\
    =& |B_{i \bar j}f^{\alpha }|^2+\frac{m-1}{m}\la Pf,\dbb \bar f \ra - \widetilde{R}_{\bar \alpha \beta \gamma \bar \sigma } f_{\bar i}^{\bar \sigma }f_{\bar j}^{\beta } (f_{i}^{\gamma }f_{ j}^{\bar \alpha }-f_{ j}^{\gamma }f_{i}^{\bar \alpha })
    -\frac{m-1}{2m}f_{00k}^{\alpha }f_{\bar k}^{\bar \alpha }.
\end{aligned}  $$
Taking integration of $\delta E $ over $M $ gives 
$$ \begin{aligned}
    -\frac{m-1}{m} \int_{M} \la Pf,\dbb \bar f \ra dV_g =& \int_{M } |B_{i \bar j}f^{\alpha }|^2 dV_g  - \int_{M}  \widetilde{R}_{\bar \alpha \beta \gamma \bar \sigma } f_{\bar i}^{\bar \sigma }f_{\bar j}^{\beta } (f_{i}^{\gamma }f_{ j}^{\bar \alpha }-f_{ j}^{\gamma }f_{i}^{\bar \alpha }) \\
    &- \frac{m-1}{2m}\int_{M} f_{00k}^{\alpha }f_{\bar k}^{\bar \alpha } dV_g.
\end{aligned}  $$
Note that 
$$ f_{\bar k k}^{\bar \alpha }-f_{k \bar k}^{\bar \alpha }=-2 \sqrt{-1}m f_{0}^{\bar \alpha } ,$$
thus,
$$ \begin{aligned}
    \int_{M} f_{00k}^{\alpha }f_{\bar k}^{\bar \alpha } dV_g=& - \int_{M} f_{00}^{\alpha }f_{\bar k k}^{\bar \alpha } dV_g\\
   =& -\int_{M} f_{00}^{\alpha }  (f_{k \bar k}^{\bar \alpha }-2m \sqrt{-1}f_{0}^{\bar \alpha }) \\
   =& -\frac{1}{2}\int_{M} f_{00}^{\alpha }(\overline{(Lf)^{\alpha }}-f_{00}^{\bar \alpha })
   +2m \sqrt{-1}\int_{M} f_{00}^{\alpha }f_{0}^{\bar \alpha }\\
   =& \frac{1}{2}\int_{M}|f_{00}^{\alpha }|^2 -\frac{1}{2}\int_{M} f_{00}^{\alpha }\overline{(Lf)^{\alpha }}
   +2m \sqrt{-1}\int_{M} f_{00}^{\alpha }f_{0}^{\bar \alpha }.
\end{aligned}  $$
Therefore,
\begin{equation}\label{pani}
    \begin{aligned}
        -\frac{m-1}{m} \int_{M} \la Pf,\dbb \bar f \ra dV_g =& \int_{M } |B_{i \bar j}f^{\alpha }|^2 dV_g  - \int_{M}  \widetilde{R}_{\bar \alpha \beta \gamma \bar \sigma } f_{\bar i}^{\bar \sigma }f_{\bar j}^{\beta } (f_{i}^{\gamma }f_{ j}^{\bar \alpha }-f_{ j}^{\gamma }f_{i}^{\bar \alpha }) \\
        &- \frac{m-1}{4m}\int_{M}|f_{00}^{\alpha }|^2 dV_g+\frac{m-1}{4m}\int_{M} f_{00}^{\alpha }\overline{(Lf)^{\alpha }}dV_g\\
       &-(m-1) \sqrt{-1}\int_{M} f_{00}^{\alpha }f_{0}^{\bar \alpha }dV_g.
    \end{aligned}
\end{equation}
Recall that    the curvature tensor $\widetilde{R}_{ \beta \bar \alpha\gamma \bar \sigma } $ is said to be strongly negative (resp. strongly semi-negative ) if
    $$ \widetilde{R}_{ \beta \bar \alpha\gamma \bar \sigma } \left( A^{\beta }\overline{B^{ \alpha }}-  C^{\beta }\overline{D^{ \alpha }} \right)\overline{ \left(A^{\sigma }\overline{B^{ \gamma  }}-  C^{\sigma  }\overline{D^{ \gamma  }}\right)}  $$
    is positive (resp. non-negative)
     for any complex numbers $A^{\alpha },B^{\alpha } ,C^{\alpha } ,D^{\alpha }  $ whenever there exists at least one pair of indices $(\alpha ,\beta )$ such that
      $A^{\beta }\overline{B^{\alpha }}-C^{\beta  }\overline{D^{\alpha }} \neq  0$   (cf. \cite{Siu1980rigid}). Evidently, strongly negative curvature (resp. strongly semi-negative curvature)  implies negative sectional curvature (resp. semi-negative sectional curvature).
    If $N $ has strongly semi-negative curvature, then 
$$ -\widetilde{R}_{\bar \alpha \beta \gamma \bar \sigma } f_{\bar i}^{\bar \sigma }f_{\bar j}^{\beta } (f_{i}^{\gamma }f_{ j}^{\bar \alpha } -f_{ j}^{\gamma }f_{i}^{\bar \alpha }) =\frac{1}{2} \widetilde{R}_{ \beta \bar \alpha \gamma \bar \sigma } (f_{\bar i}^{\bar \alpha  }f_{\bar j}^{\beta }-f_{\bar j}^{\bar \alpha  }f_{\bar i}^{\beta }) (\overline{f_{\bar i}^{\bar \gamma }f_{ \bar j}^{ \sigma  } -f_{\bar j}^{\bar \gamma }f_{\bar i}^{\sigma  }}) \geq 0.$$

Next,
we introduce the $1$-form $F=F_{\bar k} \theta ^{\bar k} $ with
$$ F_{\bar k}:= (f_{\bar j j}^{\alpha }+\frac{1}{2}f_{00}^{\alpha })f_{\bar k}^{\bar \alpha } .$$
Then 
$$ \begin{aligned}
    F_{\bar k, k}=&(f_{\bar j j k}^{\alpha }+\frac{1}{2}f_{00k}^{\alpha })f_{\bar k}^{\bar \alpha }+(f_{\bar j j}^{\alpha }+\frac{1}{2}f_{00}^{\alpha })f_{\bar k k}^{\bar \alpha }\\
    =& ((Pf)_{k}^{\alpha }- 2m \sqrt{-1}A_{k j}f_{\bar j}^{\alpha })f_{\bar k}^{\bar \alpha }+ \frac{1}{2}(Lf)^\alpha f_{\bar k k}^{\bar \alpha } .
\end{aligned}  $$
Integrating $\delta F$ on $M $ yields
\begin{equation}\label{Fdiv}
    \int_{M} \la Pf,\dbb \bar f \ra dV_g = -\frac{1}{2}\int_{M} (Lf)^\alpha f_{\bar k k}^{\bar \alpha }dV_g
 +2m \sqrt{-1}\int_{M} A_{k j}f_{\bar j}^{\alpha }f_{\bar k}^{\bar \alpha } dV_g.
\end{equation}

\begin{theorem}\label{thmfoliated}
    Let $(M^{2m+1},H,J, \theta ) $  be a compact Sasakian manifold with $m \geq 2 $, and $ (N^{2n},\widetilde{J},\widetilde{g})$ be a K\"ahler manifold with strongly semi-negative curvature.
    If $f :M \to N$ is a $\dbb $-harmonic map or a $\db $-harmonic map, then $f$ is foliated. Therefore, $f $ must be $\dbb $-pluriharmonic (that is, $f_{i \bar j}^{\alpha }=f_{\bar j i}^{\alpha }=0 $) and 
    \begin{equation}\label{curv0}
        \widetilde{R}_{ \beta \bar \alpha \gamma \bar \sigma } (f_{\bar i}^{\bar \alpha  }f_{\bar j}^{\beta }-f_{\bar j}^{\bar \alpha  }f_{\bar i}^{\beta }) (\overline{f_{\bar i}^{\bar \gamma }f_{ \bar j}^{ \sigma  } -f_{\bar j}^{\bar \gamma }f_{\bar i}^{\sigma  }}) =0.
    \end{equation}
\end{theorem}
\begin{proof}
    Suppose $f $ is $\dbb $-harmonic (the case for $\db $-harmonic map is similar). Then $(Lf)^{\alpha } =0$, or equivalently,
    \begin{equation}\label{pr38}
        f_{j \bar j }^{\alpha  }+f_{\bar j j }^{\alpha  }+f_{00 }^{\alpha  }-2m \sqrt{-1} f_{0 }^{\alpha  }=0.
    \end{equation}
     Since $M $ is Sasakian, we have $A_{ij}=0 $, and hence, $\eqref{la1} $ simplifies to 
    \begin{equation}\label{xibo1}
    \begin{aligned}
       \frac{1}{2} \lap |df(\xi )|^2=& 2\sum_{j}(|f_{0j }^{\alpha  }|^2+|f_{0\bar j }^{\alpha  }|^2) +2|f_{00}^{\alpha }|^2
       +2m \sqrt{-1} (f_{0 }^{\bar \alpha  }f_{00 }^{\alpha  }-f_{0 }^{ \alpha  }f_{00 }^{\bar\alpha  })  
       \\
        &+2 f_{0 }^{\bar{\alpha } } f_{\bar j}^{\beta }f_{j}^{\gamma }f_{0}^{\bar \sigma }\widetilde{R}_{\bar{\alpha }\beta \gamma \bar \sigma }+2 f_{\bar{j} }^{\bar{\alpha } } f_{0}^{\beta }f_{0}^{\gamma }f_{j}^{\bar \sigma }\widetilde{R}_{\bar{\alpha }\beta \gamma \bar \sigma } \\
        &-2f_{0 }^{\bar{\alpha } } f_{\bar j}^{\beta }f_{0}^{\gamma }f_{j}^{\bar \sigma }\widetilde{R}_{\bar{\alpha }\beta \gamma \bar \sigma }-2f_{0 }^{\bar{\alpha } }f_{j}^{\beta }f_{0}^{\gamma }f_{\bar j}^{\bar \sigma }\widetilde{R}_{\bar{\alpha }\beta \gamma \bar \sigma }.
    \end{aligned} 
     \end{equation}   
Therefore, by Remark  \ref{re:2.2}, integrating $\eqref{xibo1} $ over $M $ and applying integrating by parts, we have 
\begin{equation}\label{00posi}
    \begin{aligned}
        4m \sqrt{-1} \int_{M} f_{0 }^{\bar \alpha  }f_{00 }^{\alpha  }  
        +2\int_{M}|f_{00}^{\alpha }|^2 \leq 0.
    \end{aligned} 
\end{equation}

On the other hand, since $f $ is $\dbb $-harmonic, we get from   \eqref{Fdiv} that
\begin{equation}\label{prpf}
    \int_{M} \la Pf,\dbb \bar f \ra dV_g=0. 
\end{equation}
From  $\eqref{pani} $ and the curvature condition, we obtain
\begin{equation}\label{00neg}
    -\int_{M}|f_{00}^{\alpha }|^2 dV_g
           -4m\sqrt{-1}\int_{M} f_{00}^{\alpha }f_{0}^{\bar \alpha }dV_g \leq 0.
\end{equation}
Then $\eqref{00posi} $ and $\eqref{00neg} $ imply that 
$f_{00}^{\alpha } =0.$ 
Substituting it into $ \eqref{xibo1}$, we get 
$$  \frac{1}{2} \lap |df(\xi )|^2\geq 2\sum_{j}(|f_{0j }^{\alpha  }|^2+|f_{0\bar j }^{\alpha  }|^2) \geq 0 . $$
Thus, $df(\xi)=0 $ by utilizing the divergence theorem and Lemma \ref{le2.5}.

Furthermore, by substituting \eqref{prpf} and $f_{0}^{\alpha }=0$ into  \eqref{pani}, we obtain 
$$ \int_{M } |B_{i \bar j}f^{\alpha }|^2 dV_g  - \int_{M}  \widetilde{R}_{\bar \alpha \beta \gamma \bar \sigma } f_{\bar i}^{\bar \sigma }f_{\bar j}^{\beta } (f_{i}^{\gamma }f_{ j}^{\bar \alpha }-f_{ j}^{\gamma }f_{i}^{\bar \alpha }) =0.$$
Note that 
$$ -\widetilde{R}_{\bar \alpha \beta \gamma \bar \sigma } f_{\bar i}^{\bar \sigma }f_{\bar j}^{\beta } (f_{i}^{\gamma }f_{ j}^{\bar \alpha } -f_{ j}^{\gamma }f_{i}^{\bar \alpha }) =\frac{1}{2} \widetilde{R}_{ \beta \bar \alpha \gamma \bar \sigma } (f_{\bar i}^{\bar \alpha  }f_{\bar j}^{\beta }-f_{\bar j}^{\bar \alpha  }f_{\bar i}^{\beta }) (\overline{f_{\bar i}^{\bar \gamma }f_{ \bar j}^{ \sigma  } -f_{\bar j}^{\bar \gamma }f_{\bar i}^{\sigma  }}) \geq 0.$$
Thus, we get \eqref{curv0} and $B_{i \bar j}f^{\alpha }=0$. Clearly, $f_{0}^{\alpha }=0$ and $A_{ij}=0$ imply that $f_{j \bar j}^{\alpha }=f_{\bar j j}^{\alpha }=0$. 
Consequently, from the definition of $B_{i \bar j}f^{\alpha }$, we have
$$ f_{\bar j i}^{\alpha }=f_{i \bar j}^{\alpha }= \frac{1}{m}f_{k \bar k}^{\alpha }h_{i \bar j}=0. $$
This completes the proof.
\end{proof}     

Note that the rank condition in Siu's theorem mentioned in the introduction can be improved as $\rank_{\rr} (df_x) \geq 3$ at some point $x$ (cf. \cite{Jost1991nonlinear}).
By a similar argument as \cite{Siu1980rigid} and \cite{CDRY2019onharm}, we get immediately from \eqref{curv0} the following
\begin{theorem}
    Let $(M^{2m+1},H,J, \theta ) $  be a compact Sasakian manifold with $m \geq 2 $ and  $ (N^{2n},\widetilde{J},\widetilde{g})$  a K\"ahler manifold with strongly negative curvature. Suppose
     $f :M \to N$ is a $\dbb $-harmonic map   and $df$ has real rank at least 3 at some point $p \in M$. Then $f$  is either $(J, \widetilde{J})$-holomorphic or anti-$(J, \widetilde{J})$-holomorphic.
\end{theorem}
\begin{remark}
    If $f$ is $\db$-harmonic (with the other assumptions unchanged), then the conclusion remains valid.
\end{remark}

\end{document}